\documentclass[noamsfonts,reqno]{amsart}
\usepackage[margin=1in]{geometry}
\usepackage[svgnames]{xcolor}
\usepackage{graphicx}
\usepackage[osf]{libertine}
\usepackage[utf8]{inputenc}
\usepackage{textcomp}
\usepackage[varqu,varl]{inconsolata}
\usepackage{amsthm}
\usepackage[libertine,varbb,smallerops]{newtxmath}
\usepackage[cal=rsfso]{mathalpha}
\usepackage{bm}
\usepackage{mathtools}
\useosf
\usepackage{amsxtra}
\usepackage{tikz} 
\usetikzlibrary{cd}
\tikzcdset{arrow style=math font}
\usepackage{xspace} 
\usepackage{microtype}
\usepackage[%
colorlinks=true,%
linkcolor=MediumVioletRed,%
citecolor=SeaGreen,
backref=section]{hyperref}

\usepackage{comment} 
\usepackage{enumitem}

\newtheorem{theorem}{Theorem}
\newtheorem*{theorem*}{Theorem}
\newtheorem{lemma}[theorem]{Lemma}
\newtheorem{proposition}[theorem]{Proposition}

\newtheorem*{corollary*}{Corollary}

\newtheorem*{question*}{Question}

\newtheorem*{conjecture*}{Conjecture}
\newtheorem*{proposition*}{Proposition}
\newtheorem{definition}[theorem]{Definition}

\theoremstyle{remark}

\newtheorem{remark}[theorem]{Remark}
\newtheorem*{remark*}{Remark}

\newcommand{\cF}{\mathcal{F}}
\newcommand{\cG}{\mathcal{G}}
\newcommand{\cH}{\mathcal{H}}

\newcommand{\cK}{\mathcal{K}}

\newcommand{\cO}{\mathcal{O}}
\newcommand{\cP}{\mathcal{P}}
\newcommand{\cQ}{\mathcal{Q}}

\newcommand{\cS}{\mathcal{S}}

\newcommand{\cV}{\mathcal{V}}

\newcommand{\bbG}{\mathbb{G}}

\newcommand{\bbQ}{\mathbb{Q}}

\newcommand{\lto}{\longrightarrow} 
\newcommand{\iso}{\cong} 
\newcommand{\isoto}{\overset{\simeq}{\longrightarrow}} 
\newcommand{\leftiso}{\overset{\simeq}{\longleftarrow}} 

\DeclareMathOperator{\Aut}{Aut}
\newcommand{\B}{\mathrm{B}} 
\DeclareMathOperator{\Coker}{Coker}

\DeclareMathOperator{\Div}{Div}
\DeclareMathOperator{\End}{End}

\DeclareMathOperator{\Ho}{Ho} 

\DeclareMathOperator{\Hom}{Hom}
\DeclareMathOperator{\HOM}{\mathbf{Hom}}
\DeclareMathOperator{\id}{id}

\DeclareMathOperator{\Isom}{Isom}
\DeclareMathOperator{\Ker}{Ker}
\DeclareMathOperator{\Obj}{Obj}
\DeclareMathOperator{\Pic}{Pic}
\DeclareMathOperator{\Spec}{Spec}

\newcommand{\CHcat}{\mathbf{CH}} 

\newcommand{\C}{\mathsf{C}} 
\newcommand{\D}{\mathsf{D}} 
\newcommand{\T}{\mathsf{T}} 

\newcommand{\mathsc}[1]{{\normalfont\textsc{#1}}}
\DeclareMathOperator{\stCorr}{\mathsc{Corr}} 
\DeclareMathOperator{\stEnd}{\mathsc{End}} 
\DeclareMathOperator{\gerb}{\mathbf{Gerb}} 
\DeclareMathOperator{\stGerb}{\mathsc{Gerb}} 
\DeclareMathOperator{\stHom}{\mathsc{Hom}} 
\DeclareMathOperator{\tors}{\mathbf{Tors}} 
\DeclareMathOperator{\stTors}{\mathsc{Tors}} 
\DeclareMathOperator{\stOp}{\mathsc{Op}}
\newcommand{\stF}{\cF}
\newcommand{\stG}{\cG}
\newcommand{\stH}{\cH}
\newcommand{\stP}{\cP}
\newcommand{\stQ}{\cQ}
\newcommand{\stS}{\cS}
\newcommand{\stCH}{\mathcal{CH}}

\newcommand{\Br}{\mathrm{Br}}
\newcommand{\CH}{\mathrm{CH}} 

\renewcommand{\H}{\mathrm{H}} 
\newcommand{\HH}{\mathbf{H}} 
\newcommand{\R}{\mathrm{R}} 
\newcommand{\RR}{\mathbf{R}}
\newcommand{\LL}{\mathbf{L}}

\DeclarePairedDelimiter\abs{\lvert}{\rvert}

\DeclarePairedDelimiter\<{\langle}{\rangle}

\newcommand{\dirlim}{\varinjlim}
\newcommand{\et}{\mathit{\acute{e}t}}
\newcommand{\pt}{\mathit{pt}} 

\newcommand{\ie}{i.e.\xspace}
\newcommand{\eg}{e.g.\xspace}
\newcommand{\loccit}{{loc.\,cit.}\xspace}

\title{Fiber integration of gerbes and Deligne line bundles}
\author{Ettore Aldrovandi and Niranjan Ramachandran} 

\address{Ettore Aldrovandi, Department of Mathematics, Florida State University, Tallahassee, FL 32306-4510 USA.}
\email{aldrovandi@math.fsu.edu}
\urladdr{http://www.math.fsu.edu/~ealdrov}

\address{Niranjan Ramachandran, Department of Mathematics, University of Maryland, College Park, MD 20742 USA.}
\email{atma@math.umd.edu}
\urladdr{http://www2.math.umd.edu/~atma/}

\date{\today}


\begin{document}

\begin{abstract}
  Let $\pi: X \to S$ be a family of smooth projective curves, and let $L$ and $M$ be a pair of line bundles on $X$.

  We show that Deligne's line bundle $\< {L,M}$ can be obtained from the $\cK_2$-gerbe $G_{L,M}$ constructed in \cite{ER} via an integration along the fiber map for gerbes that categorifies the well known one arising from the Leray spectral sequence of $\pi$. Our construction provides a full account of the biadditivity properties of $\< {L,M}$.

Our main application is to the categorification of correspondences on the self-product of a curve.

  The functorial description of the low degree maps in the Leray spectral sequence for $\pi$ that we develop is of independent interest, and, along the way, we provide an example of their application to the Brauer group.
\end{abstract}

\maketitle


\section{Introduction}

Let $S$ be a smooth variety over a field $F$, and let $\pi:X \to S$ be a smooth projective morphism of relative dimension one.  Deligne \cite{SGA4, MR902592} has constructed a bi-additive functor of Picard categories 
\[
  \Psi_{X/S} \colon \tors_X(\bbG_m) \times \tors_X(\bbG_m) \to \tors_S(\bbG_m), \quad \Psi_{X/S}(L,M) = \<{L,M}\,,
\]
where the bi-additivity means that there are natural isomorphisms
\[
  \<{L+L', M} \isoto \<{L,M} + \<{L',M}\,,\quad \<{L,M} \isoto \<{M,L}\,.
\]
Even though there are multiple approaches to $\Psi_{X/S}$ (see \S \ref{deligne-as-det}) the  proof of bi-additivity is non-trivial in each one of them.

Let $\cK_{j}$ be the usual Zariski sheaf attached to the presheaf $U \mapsto K_j(U)$ on $X$. A basic result of Bloch-Quillen is that $H^j(X, \cK_j)$ is isomorphic to the Chow group $CH^j(X)$ of codimension-$j$ cycles on $X$.  Our main result is the following:\footnote{Theorem \ref{Main} was conjectured by M.~Patnaik \cite[Remark 21.3.2]{Patnaik}.}

\begin{theorem}
  \label{Main}
  The functor $\Psi_{X/S}$ factorizes as a composition of a bi-additive functor $\cup$ and an additive functor $\int_{\pi}$:
  \[
    \tors_X(\bbG_m) \times \tors_X(\bbG_m) \overset{\cup}{\lto}
    \gerb_X(\cK_2) \xrightarrow{\int_{\pi}}  \tors_S(\bbG_m)\,.
  \]
\end{theorem} 
In the statement, $\gerb_X(\cK_{2})$ denotes the Picard 2-category of gerbes on $X$ with band $\cK_{2}$; the bi-additivity of $\Psi_{X/S}$ is a consequence of the bi-additivity of the cup-product. The functor $\int_{\pi}$  categorifies the pushforward map
\[
    \pi_*: \CH^2(X) \to \CH^1(S)\,,
\]
and it represents \cite[XVIII \S 1.3]{SGA4} the integration of a gerbe along the fibers of $\pi$.  
Thus, $\Psi_{X/S}$ is actually a categorification of the pairing
\[
  \CH^1(X) \times \CH^1(X) \overset{\cup}{\longrightarrow} \CH^2(X)
  \overset{\pi_*}{\longrightarrow} \CH^1(S)\,.
\]
The proof of Theorem \ref{Main} is obtained by combining the following Theorems \ref{prop1}, \ref{prop2}, and \ref{prop3}:
\begin{theorem}
  \label{prop1}
  On any smooth variety $Y$ over $F$, there exists a natural bi-additive functor 
  \[
    \tors_Y(\bbG_m) \times \tors_Y(\bbG_m) \overset{\cup}{\lto}
    \gerb_Y(\cK_2).
  \]
\end{theorem}
This is essentially proved in \cite{ER}, but for the biadditivity, which we address below. Biadditivity or additivity is straightforward, but we must contend with the fact that some of the entities involved are higher categories or stacks.

Let $G_{L,M}$ be the $\cK_2$-gerbe corresponding to the cup-product of line bundles $L$ and $M$ on $X$ (viewed as $\bbG_m$-torsors). 
\begin{theorem}
  \label{prop2}
  For $\pi\colon X \to S$ as above of relative dimension one, there exists a natural additive functor 
  \[
    \int_\pi \colon 
    \gerb_X(\cK_2) \lto \tors_S(\bbG_m) \,.
  \]
\end{theorem}
The proof of Theorem~\ref{prop2} consists in writing the maps in the low degree part of the Leray spectral sequence for $\pi\colon X\to S$ directly in terms of the (higher) stacks they classify. While this can be traced back in some implicit form to \cite[\S V.3.1-2]{Giraud}, we reprise it here as we need in particular an explicit description of the functors involved. In particular, the integration map is given by taking the sheaf of connected components of the pushforward gerbe from $X$ to $S$. We describe the integration map in greater generality, by working with a general site morphism.

Finally, we have
\begin{theorem}\label{prop3} 
  One has a natural isomorphism 
  \[\int_{\pi} G_{L,M} \iso \<{L,M}\>.\]
\end{theorem}
Let $Y$ be a smooth proper variety $Y$ over $F$; let $\CH^*(Y) = \oplus_j\CH^j(Y)$ for the total Chow group of $Y$. Recall that one has a homomorphism \cite[Example 16.1.2(c)]{MR1644323}
\begin{equation}\label{chow-end}
     \CH^{\mathrm{dim}~Y}(Y \times Y) \to \End(\CH^*(Y))
\end{equation}
of rings; the ring structure on the former is given by the composition of correspondences. The categorification of (\ref{chow-end}) is of great interest.  We use Theorem \ref{Main} to provide a categorification (Theorem \ref{nuovo}) of (\ref{chow-end}) when $\dim Y =1$. It should be remarked that the problem of categorification  of (\ref{chow-end}) seems to be formidable when $\dim Y > 1$: for a surface $Y$, one needs to endow the Picard $2$-category $\stGerb_{Y\times Y}(\cK_2)$ with a ring structure.
 
Let $\CHcat^1(Y)$ denote the Picard category of line bundles on $Y$. If $C$ is a smooth projective curve over an algebraically closed field, then  $\CHcat^1(C \times C)$ can be naturally enhanced to a ring category and the natural functor $\CHcat^1(C \times C) \to \stEnd(\CHcat^1(C))$ is a functor of ring categories (Theorem \ref{nuovo}).

While there are several generalizations \cite{MR2562455, MR991974, MR962493, MR1005159, MR1085257, MR1078860, MR772054, Eriksson} of Deligne's construction, they are all restricted to line bundles or codimension one. However, Theorem \ref{Main} suggests new generalizations \cite{ER2} of Deligne's construction: if $f:Y\to S$ is smooth proper of relative dimension two, there exists a natural bi-additive functor
\[
 \Psi^2_{Y/S}: \gerb_Y(\cK_2) \times \gerb_Y(\cK_2) \lto \gerb_S(\cK_2), 
\]
which is a categorification of the pairing 
\[
  \CH^2(Y) \times \CH^2(Y) \lto \CH^4(Y) \xrightarrow{f_*} \CH^2(S)\,.
\]

\subsection*{Organization}
In section~\ref{Leray} we analyze in some detail the low-degree terms exact sequence of the Leray spectral sequence for $\pi\colon X \to S$. While this is all well known from \cite[\S V.3.1-2]{Giraud}, we expand on it as several details were famously left as an exercise (\cite[Exercice 3.1.9.2]{Giraud}). Since we describe the maps in the sequence fairly explicitly, as an example we use them to illustrate an application to the Brauer group, which is of independent interest. In section~\ref{cat-int-div} we prove Theorems~\ref{prop1} and~\ref{prop2}, and, finally, we prove Theorem \ref{prop3}, the comparison with Deligne's construction,  in section~\ref{comp-deligne}. We end with a proof of the  main result (Theorem \ref{nuovo}) about the categorification of correspondences in \S \ref{sec:categ-corr}. The requisite results from the theory of Picard categories are in \S \ref{sec:picard-stacks-endom}.

\subsection*{Notations}
For any sheaf $A$ of abelian groups on a site we denote by $\tors (A)$ the Picard category of $A$-torsors and by $\stTors (A)$ the corresponding stack. Similarly, one categorical level up, for $\gerb (A)$ and $\stGerb (A)$, which denote the 2-Picard category of $A$-gerbes and the corresponding 2-stack. For any stack $\stF$, we denote by $\pi_0(\stF)$ its sheaf of connected components and by $\mathsf{F}$ the category of its global sections, that is $\HOM(\pt,\stF)$, where $\pt$ is the terminal sheaf.

\subsection*{Acknowledgements.} Thanks to Gerard Freixas I Montplet for alerting us to \cite{MR962493}, and Sasha Beilinson for his help with understanding~\cite{MR962493} and for helping us with the proof of Proposition~\ref{beilinson-lemma}.

This work was partly supported by a Travel Award Grant from the Florida State University College of Arts and Sciences, which we gratefully acknowledge.

\section{Norms and Deligne's construction} We recall some properties of Deligne's functor $\Psi_{X/S}$ and the line bundle $\<{L,M}\>$ on $S$.
\subsection{Norms and finite maps}
Let $g: {V}\to W$ be a finite and flat morphism of varieties. Given a line bundle $L$ on $V$, its norm (relative to $g$) is a line bundle $N_{V/W}(L)$ on $W$. One has an additive functor of Picard categories \cite[\S 7.1]{MR902592} 
\[
    N_{V/W}: \tors_V{\mathbb G_m} \lto \tors_W{\mathbb G_m}\,.
\]
\subsection{Characterization of Deligne's functor $\Psi_{X/S}$}\label{deligne-as-det}
Let $D\subset X$ be an effective relative Cartier divisor \cite[Tag 056P]{stacks-project} of $\pi:X \lto S$. Namely, $D$ is an effective Cartier divisor on $X$ and the induced morphism $\pi:D\to S$ is finite and flat. For any line bundle $M$ on $X$, the norm $N_{D/S}(M)$ is a line bundle on $S$.  Deligne's construction $\Psi_{X/S}$ is characterized by \cite[XVIII 1.3.16]{SGA4}: (i) functoriality, and (ii) for any section of $L$ with zero set an effective Cartier divisor $D$ on $X$, a canonical isomorphism  
\begin{equation}\label{norm-finite}
  N_{D/S} \big(M~\big|_D~\big) \iso \<{L,M} \>\,.
\end{equation}
Another approach to $\Psi_{X/S}$ from \cite[XVIII 1.3.17.2]{SGA4} is the following: if $D$ and $E$ are effective relative Cartier divisors on $X$, then 
\begin{equation}
    \<{\cO (D), \cO (E)}\> \iso \det \RR\pi_*(\cO (D) \overset{\LL}{\otimes} \cO (E)).
\end{equation}

\section{Fiber Integration of gerbes and the Leray spectral sequence}
\label{Leray}

Let $A$ be an abelian sheaf on $X$. The spectral sequence 
\begin{equation}
  \label{leray}
  E^{i,j}_2 = \H^i (S, \R^j\pi_*A) \Rightarrow \H^{i+j}(X, A)
\end{equation}
has as low-term exact sequence \cite[Appendix II, page 309]{MilneEC} 
\begin{equation}\label{lowterm}
  0 \lto E^{1,0}_2 \lto  E^1 \lto E^{0,1}_2 \lto E^{2,0}_2
  \lto E^2_1 \lto E^{1,1}_2\,,
  \end{equation}
where
\[
  E^1 = \H^1(X,A)\,, \quad
  E^2_1 = \Ker( \H^2(X, A) \lto \H^0(S, \R^2\pi_*A))\,,
\]
and, of course, $\H^0(S, \R^2\pi_*A) = E^{0,2}_2$.

The maps above arise from functors between categories of torsors and gerbes, as shown in~\cite[pp.~324--327]{Giraud}. For our own purposes, and also to rephrase the arguments in loc.~cit.\ in a more transparent way, we turn to an explicit description of these functors.

Our arguments below (in the Zariski topology) are easily seen to be also valid in the \'etale or analytic topology. In fact, at the beginning they are valid for any morphism between sites whose underlying functor is assumed for simplicity to preserve finite limits, and we shall begin our discussion in such generality.

\subsection{Site morphisms, push-forwards and pullbacks of stacks }
\label{sec:push-forw-pullb}

Let $\pi\colon \D\to \C$ be a morphism of (small) sites. We let $u=\pi^{-1}\colon \C \to \D$ denote the underlying functor. Thus, $\pi$ is a morphism of sites if composition along $u$ preserves sheaves, and this operation has a left adjoint that is exact \cite{SGA4}; this is implied by the property that $u$ preserves coverings and if both $\C$ and $\D$ have finite limits $u$ preserves them \cite{JardineLHT}, \cite[\href{https://stacks.math.columbia.edu/tag/00X0}{Tag 00X0}]{stacks-project}.

Let $\stG$ be a category over $\D$. Its \emph{push-forward} $\pi_*\stG$ along $\pi$ is defined by
\begin{equation*}
  \pi_*\stG = \C \times_{\D} \stG
\end{equation*}
as a category over $\C$ via the first projection. It is fibered (resp. a stack) if so is $\stG$. On the other hand, let $p\colon \stF\to \C$ be a stack. The inverse image $\pi^*\stF$ is a pair $(\stF',\phi)$, where $\stF'$ is a stack over $\C$, and $\phi\colon \stF\to \pi_*\stF'$ a stack morphism such that, for any stack $\stG$ over $\D$, the following composite functor
\begin{equation*}
  \HOM_\D(\stF',\stG) \lto \HOM_\C(\pi_*\stF',\pi_*\stG) \lto \HOM_\C(\stF,\pi_*\stG)\,,
\end{equation*}
is an equivalence of categories \cite[Déf.\ 3.2.1]{Giraud}. Here $\HOM$ denotes the category of stack morphisms. Thus, the inverse image is truly only defined up to equivalence.

While specific formulas to compute a model of $\pi^*\stF$ do exist \cite[\href{https://stacks.math.columbia.edu/tag/04WJ}{Tag 04WJ}]{stacks-project}, the universal property is sufficient to characterize its connected components.
Recall that $\pi_0(\stF)$ is the sheaf corresponding to the presheaf of connected components: to any object $U\in \C$ it assigns the set of connected components $\pi_0(\stF_U)$ of the fiber category $\stF_U$ \cite[Chap.\ 7]{BreenAst}. (In ref.\ \cite[n.\  2.1.3.3]{Giraud} this is the ``sheaf of maximal sub-gerbes of $\stF$.'') We have \cite[Prop.\ 2.1.5.5 (iii)]{Giraud} an isomorphism of sheaves over $\D$:
\begin{equation*}
  \pi_0(\pi^*\stF) \isoto \pi^*(\pi_0(\stF)) \,.
\end{equation*}
This follows from the fact that if $x,y$ are any two objects of $\stF$ over $U\in \C$, then there is a sheaf isomorphism
\begin{equation*}
  \pi^*\Hom_{\stF}(x,y) \isoto \Hom_{\stF'}(x',y')\,,
\end{equation*}
where the objects $x',y'$ of $\stF'_{\pi^{-1}(U)}$ are constructed via the above universal property (ibid.). As a consequence, since a gerbe is locally connected, we have that the inverse image of a gerbe is a gerbe \cite[Cor.\ 2.1.5.6]{Giraud}. In fact, assuming, as we shall do in later sections, that $\stF$ has band $A$, for an abelian sheaf $A$ over $\C$, then $\pi^*\stF$ has band $\pi^*A$.

On the other hand, even if $\stG\to \D$ is a gerbe, its push-forward will not necessarily be so. In other words, $\pi_0(\pi_*(\stG))$ may turn out to be a nontrivial sheaf over $\C$. More precisely, we have the following statement.
\begin{lemma}[\protect{\cite[Exercice 3.1.9.2]{Giraud}}]
  \label{lem:exercise}
  Let $\pi\colon \D\to \C$ be a site morphism as above. Let $\stG$ be an $A$-gerbe on $\D$, where $A$ is an abelian sheaf. Then $\pi_0(\pi_*(\stG))$ is a pseudo $\R^1\pi_*A$-torsor. It is a torsor if and only if the class $[\stG]\in \H^2(\D,A)$ lies in the kernel of the map $\H^2(\D,A)\to \H^0(\C,\R^2\pi_*A)$, and hence in the term denoted $E^2_1$ above.
\end{lemma}
An $A$-gerbe $\stG$ is {\it{horizontal}} if its class $[\stG]$ lies in $E^2_1$. If $\stG$ is horizontal, then $\pi_0(\pi_*(\stG))$ is an $\R^1\pi_*A$-torsor. 

By $\H^i(\C,-)$ (same for $\D$) we denote the cohomology of the terminal sheaf $\pt$. In the concrete case of the Zariski sites, where $\pt$ is represented by the site's terminal object, this reduces to the groups considered at the beginning of this section.

\begin{definition}\label{horizontal}
Let $\stGerb_X(\cK_2)'$ be the full sub(2-)category of $\stGerb_X(\cK_2)$ consisting of horizontal gerbes. The functor 
\[
    \Theta_{\pi}:\gerb_X(\cK_2)'\to \tors_S(\R^1\pi_*\cK_2)
\]
sends a gerbe $\stG$ to $\pi_0(\pi_*(\stG))$. 
\end{definition}
We write $\theta:E^2_1\to E^{1,1}_2$ for the induced map. (In the previous statement, as well as in several that follows, the relevant (2-)categories can be upgraded to the correspoonding (2-)stacks.)

\subsection{Proof of Lemma~\protect~\ref{lem:exercise}}
\label{sec:médaille_de_chocolat}
This section is devoted to a complete proof of Lemma~\ref{lem:exercise}. Several points of the proof will be explicitly needed in sections~\ref{cat-int-div} and~\ref{comp-deligne} below.\footnote{As all the details are famously not available in the original reference as well as in the literature, we felt compelled to include them here.}

It is convenient to express the sites' topologies in terms of local epimorphisms \cite{SGA4,KS}, and take hypercovers of those, in particular Čech nerves. For simplicial objects we use the ``opposite index convention'' \cite{Duskin2001} (and reverse the order of the maps for cosimplicial ones) when pulling back by simplicial maps: $d^*_i(-) = (-)_{[n]\setminus i}$, where $[n]$ is the ordinal $[n]=\{0 < 1 < \dots < n\}$. 

\subsubsection{Objects with operators}
\label{sec:objects-with-oper}

Let $\stF$ be a stack over a site $\C$, and let $G$ be a sheaf of groups over $\C$. The stack $\stOp(G,\stF)$ has objects the pairs $(x,\eta)$, where $x\in \stF_U$, and $\eta\colon G\vert_U\to \Aut_U(x)$. Morphisms from $(x,\eta)$ to $(y,\theta)$ are arrows $\alpha\colon x\to y$ in $\stF_U$ compatible with the structure: $\alpha \circ \eta (g) = \theta(g)\circ\alpha$, for all sections $g\in G\vert_U$.
\begin{lemma}[\protect{\cite[III № 2.3]{Giraud}}]
  \label{lem:twist}
    There is a stack morphism $t\colon \stTors(G)\times_\C \stOp(G,\stF)\to \stF$.
\end{lemma}
This ``twisting'' morphism assigns to each pair $(P,(x,\eta))$ over $U\in\C$ an object of $\stF_U$, variously denoted as $\prescript{P}{}x$ or $P\wedge^G x$.
\begin{proof}
If $P=G$, the trivial $G$-torsor, we set $\prescript{G}{}x = x$. To a morphism $(g,\alpha): (G,(x,\eta)) \to (G,(x',\eta'))$ (here $g\in G$ is identified with an automorphism of the trivial torsor) we assign the morphism $\prescript{G}{}x \to \prescript{G}{}x'$ given by $\alpha\circ \eta(g) = \eta'(g)\circ \alpha$. 
In general, we regard $P\in \stTors(G)\vert_U$ and $x\in \stF_U$ as defined by descent data relative to  an acyclic fibration $\epsilon \colon V_\bullet\to U$ covering $U$. The pullbacks $\epsilon^*x$ and $\epsilon^*P \iso G\vert_{V_0} $ to $V_0$ are glued over $V_1$ by isomorphisms
\begin{equation*}
  \alpha \colon x_1 \to x_0\,, \qquad g \colon G \to G\,,
\end{equation*}
where $g\in G(V_1)$ is an isomorphism between trivial $G\vert_{V_1}$-torsors, satisfying the cocycle conditions $\alpha_{02} = \alpha_{01}\circ \alpha_{12}$ and $g_{01}g_{12} = g_{02}$ over $V_2$. That $(x,\eta)$ is an object of $\stOp(G,\stF)$ is expressed by the condition $\alpha \circ \eta_1(-) = \eta_0(-) \circ \alpha$ over $V_1$.

Pullbacks to $V_2$ along the face maps $d_0,d_1,d_2$ yield  morphisms $(g_{ij},\alpha_{ij}): (G,(x_j,\eta_j)) \to (G,(x_i,\eta_i))$, $0\leq i < j \leq 2$, in $\stOp(G,\stF)_{V_2}$ such that $\alpha_{ij}\circ \eta_j(g_{ij}) = \eta_i(g_{ij})\circ \alpha_{ij}$, in addition to the other cocycle conditions. Then we have
\begin{equation*}
  \alpha_{02} \circ \eta_2(g_{02}) =
  \alpha_{01}\circ \alpha_{12} \circ \eta_2(g_{01})\circ \eta_2(g_{12})
  = (\alpha_{01}\circ \eta_1(g_{01})) \circ (\alpha_{12}\circ \eta_2(g_{12}))\,,
\end{equation*}
showing the gluing data $\alpha \circ \eta_1(g) = \eta_0(g) \circ \alpha \colon x_1\to x_0$ satisfy the cocycle identity and therefore define an object $\prescript{P}{}x$ of $\stF_U$.
\end{proof}
The most important properties of the twisted objects are listed in the following lemma implicit in \cite{Giraud}, where for any two objects $x,y\in \stF_U$ we denote by $\Isom_U(x,y)$ the sheaf of isomorphisms from $x$ to $y$. Note that $\Isom_U(x,y)$ is a right $\Aut_U(x)$-torsor. (In fact it is an $(\Aut_U(y),\Aut_U(x))$-bitorsor, but we shall not need this fact.)
\begin{lemma}[\cite{Emsalem2017,MR2362847}]
  \label{lem:tors}
  Let $\stF$ be a stack and $G$ a sheaf of groups on $\C$. Let $\stOp(G,\stF)$ the stack of objects with $G$-action. The twisting morphism $t$ has the following properties:
  \begin{enumerate}
  \item\label{item:1} If $P\in \stTors(G)_U$, and $x\in \stF_U$, then
    \begin{math}
      \Isom_U(x, P\wedge^G x) \iso P\wedge^G \Aut_U(x)\,;
    \end{math}
  \item\label{item:2} If $P=\Isom_U(y,x)$, then there is a canonical isomorphism $P\wedge^{\Aut_U(y)} y\isoto x$, where the twisting arises from the stack $\stOp(\Aut_U(y),\stF\vert_U)$ over $\C/U$. 
  \end{enumerate}
\end{lemma}
\begin{proof}
  Let $(y,\theta)$ be another object of $\stOp(G,\stF)$ over $U$. Using the same notation as in Lemma~\ref{lem:twist} for descent data relative to $V_\bullet \to U$, a morphism $Q\wedge^G y \to P\wedge^G x$ of twisted objects over $U$ corresponds to a morphism $\lambda \colon \epsilon^*y \to \epsilon^*x$ over $V_0$ such that the diagram over $V_1$
  \begin{equation}
    \label{eq:1}
    \begin{tikzcd}
      y_1 \ar[r,"\lambda_1"] \ar[d,"\beta \circ\theta_1 (h)"'] &
      x_1 \ar[d,"\alpha\circ\eta_1(g)"] \\
      y_0 \ar[r,"\lambda_0"']        & x_0
    \end{tikzcd}
  \end{equation}
  commutes. Here $\beta$ and $h$ represent the descent data and cocycle for $y$ and the $G$-torsor $Q$, respectively. 

  In particular, if $y=x$ and $Q$ is the trivial torsor, we get the simpler relation
  \begin{equation*}
    \alpha \circ \eta_1(g) \circ \lambda_1 = \lambda_0\circ \alpha\,.
  \end{equation*}
  Rewriting it in the more suggestive way
  \begin{equation*}
    \eta_1(g) \circ \lambda_1 = \alpha^{-1}\circ \lambda_0\circ \alpha
  \end{equation*}
  shows that the $\lambda$ defines a section of $P\wedge^G \Aut_U(x)$, proving the first point.

  If $P=\Isom_U(x,y)$, $\lambda \colon \epsilon^*y \to \epsilon^*x$ provides a section of $P$ over $V_0$. As $\lambda$ does not necessarily descend to $U$, the two pullbacks $\lambda_0$ and $\lambda_1$ to $V_1$ are related by a diagram of the form
  \begin{equation*}
    \begin{tikzcd}
      y_1 \ar[r,"\lambda_1"] \ar[d,"\bar h"'] & x_1 \ar[dd,"\alpha"] \\
      y_1 \ar[d,"\beta"'] \\
      y_0 \ar[r,"\lambda_0"'] & x_0
    \end{tikzcd}
  \end{equation*}
  for an appropriate $\bar h\in \Aut(y_1) = \Aut_U(y)(V_1)$. Comparing with~\eqref{eq:1} (taking $\theta=\id$) shows these data descend to an isomorphism $P\wedge^{\Aut_U(y)} y\isoto x$, as wanted.
\end{proof}

$\stF$ is an abelian gerbe with band $A$ if and only if there is a \emph{canonical morphism} $\stF \to \stOp(A;\stF)$, because, in such case, the correspondence
\begin{equation*}
  x\in \stF_U \rightsquigarrow
  \eta_x\colon A\vert_U\isoto \Aut_U(x)
\end{equation*}
is functorial: the diagram
\begin{equation*}
  \begin{tikzcd}
    & A\vert_U \ar[dl,"\eta_x"'] \ar[dr,"\eta_y"] \\
    \Aut_U(x) \ar[rr,"\alpha_*"'] && \Aut_U(y)
  \end{tikzcd}
\end{equation*}
commutes whenever $\alpha\colon x\to y$ \cite[Def.\ 2.9]{BreenAst}. (Note that the above diagram embodies a morphism of $\stOp(A;\stF)$.) Therefore there is a \emph{canonical twisting action} $(P,x) \rightsquigarrow {}^Px$ resulting from the composite morphism
\begin{equation*}
  \stTors(A) \times_\C \stF \lto \stTors(A)\times_\C \stOp(A; \stF) \lto \stF\,.
\end{equation*}

\subsubsection{The pushforward}
\label{sec:pushforward}

Let us return to the situation of the site morphism $\pi\colon \D\to \C$. Recall that $u\colon \C\to \D$ is the underlying functor of $\pi$.

Let $A$ be an abelian sheaf and  $\stG$ be an $A$-gerbe over $\D$. It is convenient to identify the band with the automorphism sheaves. In this way, Lemma~\ref{lem:tors}, statement~(\ref{item:1}), simply becomes $\Isom_U(x,P\wedge^A x)\iso P$.

The action $\pi_0(\pi_*\stG) \times \R^1\pi_*A \to \pi_0(\pi_*\stG)$ is induced by the twisting action of $\stTors(A)$ on $\stG$ on $\D$: if $x\in \stG_{u(U)}$ represents a section of $\pi_0(\pi_*\stG)$, and $P\in \stTors(A)_{u(U)}$ represents a class of $\R^1\pi_*A\,(U)$, we let the result of the action be the connected component of the object $P\wedge^A x\in \stG_{u(U)} \iso \pi_*(\stG)_U$.  This action is free, because if $P\wedge^Ax \iso x$, then by Lemma~\ref{lem:tors}~(\ref{item:1}) $\Hom_U(x,x\wedge^A x) \iso P$ has a global section, hence $P\iso A\vert_{u(U)}$.

The action is also transitive. Indeed, if the objects $x,y\in \stG_{u(U)} \iso \pi_*(\stG)_U$ represent two sections of $\pi_0(\pi_*\stG)$, by Lemma~\ref{lem:tors}~(\ref{item:2}) we have $y \iso P \wedge^A x$, where $P=\Hom_{u(U)}(x,y)$. Therefore the section of $\pi_0(\pi_*\stG)$ over $U$ defined by $y$ is obtained from that defined by $x$ via the action of the section of $\R^1\pi_*A\, (U)$ determined by $P$, as wanted. Thus, $\pi_0(\pi_*\stG)$ is a pseudo-torsor.

Let $U$ be an object of $\C$, and denote by $U'=u(U)$ the corresponding object of $\D$. As a gerbe, $\stG$ is locally nonempty, hence there will be a local epimorphism $V'\to U'$ covering $U'$ with an object $x\in \stG_{V'}$. The object $x$ ought to be seen as a trivialization of the restriction $\stG\vert_{U'}$, whose characteristic class is an element of $\H^2(U',A\vert_{U'}) \iso \R^2\pi_*A \,(U)$. (In effect, this class can be calculated by computing the $2$-cocycle determined by $x$ via a hypercovering $V'_\bullet\to U'$ in the usual way \cite{BreenAst}.) If it is zero, then $\stG_{U'}$ has a global object (that is, $x$ descends to an object over $U'$), which then provides a section of $\pi_0(\pi_*\stG)$ over $U\in \C$. Clearly, if $\R^2\pi_*A$ vanishes, this argument shows that $\pi_0(\pi_*\stG)$ is locally nonempty. On the other hand, if $\pi_0(\pi_*\stG)$ is locally nonempty, for every object $U$ we can find a local epimorphism $V\to U$ such that $\stG_{u(V)}$ has an object and therefore $\H^2(u(V),A\vert_{u(V)}) = 0$. Now writing $\R^2\pi_*A (U) = \lim_{[V_\bullet \to U]} \H^2(u(-),A\vert_{u(-)})$ we get $\R^2\pi_*A =0$. This finishes the proof of Lemma~\ref{lem:exercise}.

\subsection{Maximal subgerbes and pullbacks}
\label{sec:maxim-subg-pullb}

The following extra facts (``tautologies'' in \cite[V~№ 3.1.8]{Giraud}) are going to be helpful. Recall that we have a site morphism $\pi\colon \D\to \C$. Following Giraud, let us say that a gerbe $\stG$ over $\D$ \emph{comes from} a gerbe on $\C$, if there is a gerbe $\stF$ on $\C$ and a morphism of gerbes $m\colon \pi^*\stF\to \stG$ over $\D$.
\begin{lemma}
  \label{lem:tautology}
  The gerbe $\stG$ on $\D$ comes from a gerbe on $\C$ if and only if the sheaf $\pi_0(\pi_*\stG)$ admits a section.
\end{lemma}
\begin{proof}
  If $\stG$ comes from a gerbe on $\C$, let $\stF$ be such a gerbe and $m\colon \pi^*\stF\to \stG$ the corresponding morphism. By adjunction (cf.\ the universal property that defines the operation $\pi^*$, sect.~\ref{sec:push-forw-pullb}) we obtain a morphism $n\colon \stF \to \pi_*\stG$. Since $\stF$ is a gerbe, for the sheaf of connected components we get $\pi_0(n)\colon \pt\to \pi_0(\pi_*\stG)$, hence a section of $\pi_0(\pi_*\stG)$.

  Conversely, if $\pi_0(\pi_*\stG)$ has a section, say $\xi\colon \pt \to \pi_0(\pi_*\stG)$, define $\stF = \pt\times_{\pi_0(\pi_*\stG)} \pi_*(\stG)$, which is a gerbe on $\C$, and $n\colon \stF\to \pi_0(\pi_*\stG)$ as the second projection. The latter is by construction fully faithful, hence, again by adjunction, we have the morphism $m\colon \pi^*(\stF) \to \stG$, and so $\stG$ comes from a gerbe on $\C$.
\end{proof}
\begin{remark}
  In the previous proof we have used the well known fact (but, again, ultimately due to Giraud \cite[III Prop.\ 2.1.5.3]{Giraud}) that any stack $\stS$ the projection $\stS\to \pi_0(\stS)$ makes it a gerbe on the sheaf of its connected components. For a section $\xi\in \pi_0(\pi_*\stG)$, the pullback $\xi^*(\stS)$ is the corresponding \emph{maximal sub-gerbe.}
\end{remark}
\begin{remark}
  Using Lemma~\ref{lem:exercise}, we see that Lemma~\ref{lem:tautology} is equivalent to the exactness at $\H^2(\C,\pi_*A)$ of the low term sequence
  \begin{equation*}
    \cdots \lto \H^0(\C,\R^1\pi_*A) \lto \H^2(\C,\pi_*A) \lto
    \H^2(\D,A)' \lto \cdots
  \end{equation*}
  arising from the Leray spectral sequence we recalled above, where we set $\H^2(\D,A)' = E^2_1$.
\end{remark}

\subsection{Interpretation of the maps}
\label{sec:interpretation-maps}
  
\subsubsection{The map $E^{1,0}_2 \to  E^1$}\label{sec:map-e1-0_2} This is the obvious pull-back map $\H^1(S, \pi_*A) \to \H^1(X,\pi^*\pi_*A) \to \H^1(X, A)$ using the natural adjunction $\pi^*\pi_* A \to A$ of sheaves on $X$.  This is just the composite functor 
\[
  \stTors_S(\pi_*(A)) \overset{\pi^*}{\lto} \stTors_X(\pi^*\pi_*(A))
  \lto \stTors_X(A)
\]
for the corresponding gerbes.

\subsubsection{The map $E^1\to E^{0,1}_2$}
\label{sec:map-e0-1_2}
Using that $\R^1\pi_*A$ is the sheaf associated to $U \rightsquigarrow \H^1(\pi^{-1}(U),A)$, the map is obtained by considering the class of an object $P$ of $\stTors_X(X)$ under $P \mapsto \pi_*P$.

\subsubsection{The map $E^{0,1}_2 \to E^{2,0}_2$}
This is the transgression map relative to the standard sequence arising from an injective resolution $0\to A \to I^\bullet$ on $X$. Then it is standard that
\begin{equation*}
  0 \lto \pi_*A \lto \pi_*I^0 \lto Z^1(\pi_*I^\bullet) \lto \R^1\pi_*A \lto 0\,.
\end{equation*}
Viewing it as the splicing of two short exact sequences
\begin{equation*}
  0 \lto \pi_*A \lto \pi_*I^0 \lto C \lto 0\,,\qquad
  0\lto C \lto Z^1(\pi_*I^\bullet) \lto \R^1\pi_*A \lto 0\,,
\end{equation*}
the transgression map is the composite
\begin{equation*}
  \H^0(S,\R^2\pi_*A) \lto \H^1(S,C) \lto \H^2(S,\pi_*A)\,.
\end{equation*}
The latter is obtained by taking the global objects of the composite 2-functor:
\begin{equation*}
  A \lto \stTors_S(C) \lto \stGerb_S(\pi_*A)\,.
\end{equation*}
The map on the right is the well known classifying map of the extension of $\R^1\pi_*A$ by $C$ above \cite[V~№ 3.2]{Giraud} (see also \cite{ER}).

\subsubsection{The map $E^{2,0}_2 \to E^{2}_1$}
Analogously to \ref{sec:map-e1-0_2}, we have the composite 2-functor
\begin{equation*}
  \stGerb_S(\pi_*A) \lto \stGerb_X(\pi^*\pi_*A) \lto \stGerb_X(A)\,,
\end{equation*}
where the arrow on the right is ``change of band'' functor along $\pi^*\pi_* A \to A$.
Taking isomorphism classes in the global fibers gives the composite $\H^2(S,\pi_*A) \to \H^2(X, \pi^*\pi_*A) \to \H^2(X,A)$. Now, thanks to Lemma~\ref{lem:tautology} the image is actually in $E^2_1$.

\subsubsection{The map $\theta\colon E^2_1 \to E^{1,1}_2$}
Let $\stG$ be a stack on $X$ and consider the correspondence $\stG \leadsto \pi_0(\pi_*\stG)$. This correspondence is easily seen to be a functorial one between the homotopy category of stacks—as it identifies two naturally isomorphic stack morphisms—on $X$ to that of sheaves on $S$. Therefore, by Lemma~\ref{lem:exercise}, and the subsequent sections, it reduces to a functor $\Ho (\gerb_X(A)') \to \tors_S(\R^1\pi_*(A))$, where $\gerb_X(A)'$ denotes the subcategory of those $A$-gerbes whose fiber categories over opens of the form $\pi^{-1}(U)$, for every sufficiently small open neighborhood $U\subset S$ around every point of $S$, are not empty. By taking classes, we get the map.

\subsection{Application: Brauer groups}\label{Brauer}
The map $\theta: E^2_1 \to E^{1,1}_2$ above plays a very important role in many arithmetical applications \cite{Kai, Lichtenbaum, Skorobogatov}. To recall this, let $F$ be a perfect field and consider the \'etale sheaf $\bbG_m$ on $T = \Spec F$.  Fix an algebraic closure  $\bar{F}$  of $F$ and let $\bar{T} =\Spec\bar{F}$; write $\Gamma$ for the Galois group of $\bar{F}$ over $F$. Let $g:Y\to T$ be a smooth proper map and $\bar{Y} = Y\times_T \bar{T}$.

The group $\H^2_{\et}(Y, \bbG_m)$ is the Brauer group $\Br(Y)$.  The group $E^2_1$ is the relative Brauer group $\Br(\bar{Y}/{Y})$, namely, the kernel of the map $\Br(Y) \to \Br(\bar{Y})$. The map $E^2_1 \to E^{1,1}_2$ then becomes the  map
\[
  \theta \colon \Br(\bar{Y}/Y) \lto \H^1_{\et}(T, \R^1g_* \bbG_m)\,,
\]
arising in several contexts. For instance, for any elliptic curve $E$ over $F= \bbQ$, the map $\theta$ gives the well known isomorphism
\[
  \Br(E) \simeq \H^1(\bbQ, E(\bar{\bbQ})) =
  \H^1_{\et}(\Spec\bbQ, E)\,.
\]
The following explicit description of $\theta$ seems to be missing in the literature: Given an element $\alpha$ of $ \Br(\bar{Y}/Y)$, pick a $\bbG_m$ gerbe $G$ on $Y$ representing $\alpha$. By definition, the base change $\bar{G}$ on $\bar{Y}$ is trivial. Fix an equivalence $f\colon  \bar{G} \simeq \tors_{\bar{Y}}(\bbG_m)$. Then, given any $\sigma \in \Gamma$, the gerbe $\sigma^*\bar{G}$ is equivalent to $\bar{G}$ as $G$ comes from $Y$. Write $f_{\sigma}$ for the resulting equivalence of $\tors_{\bar{X}}(\bbG_m)$:
\[
  f_{\sigma}\colon \tors_{\bar{Y}}(\bbG_m) \overset{f}{\longleftarrow}
  \bar{G} \simeq \sigma^*\bar{G} \xrightarrow{\sigma^*f} \sigma^* \tors_{\bar{Y}}(\bbG_m) = \tors_{\bar{Y}}(\bbG_m)\,.
\]
Any self-equivalence  \cite[\S5.1]{Milne}
\[
  \tors_{\bar{Y}}(\bbG_m) \simeq \tors_{\bar{Y}}(\bbG_m)
\]
is a translation by a fixed $\bbG_m$-torsor $L$, namely the self-equivalence is of the form $(-) \mapsto (-) + L$.  Therefore, if $L_{\sigma}$ is the $\bbG_m$-torsor on $\bar{Y}$ corresponding to $f_{\sigma}$, then the map $\sigma \mapsto L_{\sigma}$ represents the element $\theta(\alpha)$ of $\H^1_{\et}(T, \R^1g_*\bbG_m)$.

\section{Gerbes and categorical intersection of divisors}
\label{cat-int-div}
In this section, we prove Theorems \ref{prop1} and \ref{prop2}. 

Let $Y$ be a smooth variety over $F$. Let $\cK_i$ denote the Zariski sheaf associated with the presheaf $U \mapsto K_i(U)$. 

\subsection{Heisenberg groups}

For any pair abelian sheaves $A$ and $B$ on $Y$, we have constructed~\cite{ER} a Heisenberg sheaf $H_{A,B}$ (of nilpotent groups) which fits into an exact sequence
\begin{equation}
  \label{eq:heisenberg}
  0 \lto A\otimes B \lto H_{A,B} \lto A\times B \lto 0\,,
\end{equation}
providing a categorification of the cup-product 
\begin{equation}
  \label{cupproduct}
  \H^1(Y,A) \times \H^1(Y,B) \lto \H^2(Y, A\otimes B)
\end{equation}
in the following manner. Given an $A$-torsor $P$ and a $B$-torsor $Q$, the $A\times B$-torsor $P\times Q$ can be lifted locally to a $H_{A,B}$-torsor in several ways. These local lifts assemble to a $A\otimes B$-gerbe $G_{P,Q}$.

Much like the bulk of section~\ref{Leray}, we can formulate the result we need in much greater generality. As in \cite[\S 3]{ER}, we assume $A$ and $B$ are abelian objects of a topos $\T$. (In the applications, we assume $\T$ to be the topos of sheaves on the Zariski, or other relevant topology, of the scheme.) Recall from \loccit
that the Heisenberg group $H_{A,B}$ is defined by the group law:
\begin{equation*}
  (a,b,t)\, (a',b',t') = (aa',bb',t + t' + a\otimes b')\,,
\end{equation*}
where $a,a'$ are sections of $A$, $b,b'$ of $B$, and $t,t'$ of $A\otimes B$. The extension~\eqref{eq:heisenberg} is set-theoretically split, \ie there is a section of the underlying map of sheaves of sets. The map
\begin{equation}
  \label{cocycle}
  f \colon (A\times B)\times (A\times B) \lto A\otimes B,
  \quad
  f(a,b,a',b') = a\otimes b',
\end{equation}
is a cocycle representing the class of the extension in $\H^2(\B_{A\times B}, A\otimes B) \iso \HH^2(K(A \times B,1), A\otimes B)$, where on the left we have the cohomology of classifying topos \cite{Giraud}, and on the right that of the corresponding Eilenberg-Mac~Lane simplicial object of $\T$. In fact these cohomologies are in turn isomorphic to $[ K(A \times B,1) , K(A \otimes B,2) ]$, the hom-set in the homotopy category \cite{MR516914,MR0491680}, and the cocycle $f$ coincides with the only non-trivial component of the characteristic map \cite[Prop.\ 3.4]{ER}.

\begin{proposition}
  \label{prop:biadditive}
  The functor
  \[
    c_{A,B} \colon \stTors (A) \times \stTors (B) \lto
    \stGerb(A\otimes B)\,,\quad P\times Q \longmapsto G_{P,Q}
  \]
  is bi-additive. On $\pi_0$, it induces the cup-product map \eqref{cupproduct}, upon choosing $\T = Y_{\mathrm{Zar}}\sptilde$.
\end{proposition} 

\begin{remark}[On bi-additivity]
  \label{rem:biadditivity}
  Note that for any (abelian) band $L$, $\stGerb (L)$ is really a 2-stack. Hence the notion of 2-additivity should be updated with appropriate 2-coherence data from higher algebra. This is both outside the scope of this note and inconsequential in the case at hand. Alternatively, we can mod out the 2-morphisms and consider $\stGerb (L)$ as a Picard 1-stack of $\T$. Thus, bi-additivity consists of the data of functorial equivalences
  \begin{align*}
    c_{A,B} (P_1 + P_2, Q) &\isoto c_{A,B} (P_1, Q) +c_{A,B}(P_2, Q) \\
    c_{A,B} (P, Q_1 + Q_2) &\isoto c_{A,B} (P, Q_1) +c_{A,B}(P, Q_2)
  \end{align*}
  subject to the condition that decomposing $c_{A,B} (P_1 + P_2, Q_1 + Q_2)$ according to the two possible ways determined by the above morphisms gives rise to a commutative (or commutative up to coherent 2-isomorphism) diagram. This would be exactly the kind of diagram familiar from the theory of biextensions \cite{MR0354656-VII,MR823233} (see also \cite{MR1114212}).
\end{remark}
\begin{proof}[Proof of Proposition~\ref{prop:biadditive}]
  Bi-additivity is essentially already implied by the fact that the cocycle representing the class of the extension is the tensor product, which is bilinear. It is best to look at this in the universal case, namely over $K(A \times B, 1)$—the rest follows by pullback—where the bilinearity of the tensor product has the following interpretation.

  By functoriality, from the group operation ($A$ is an abelian object) $+_A\colon A \times A \to A$ we get the map $+_A\colon K(A \times A \times B, 1) \to K(A \times B, 1)$ corresponding to the Baer sum of torsors. Its composition with the characteristic map $c \colon K(A\times B,1)\to K(A\otimes B,2)$ equals in the homotopy category the sum $c_1 + c_2$, where $c_i$, $i=1,2$, is the composition
  \begin{equation*}
    K(A \times A \times B,1) \overset{p_i}{\lto} K(A \times B,1) \overset{c}{\lto} K(A \otimes B, 2)\,;
  \end{equation*}
  the first map is induced by the projection onto the first (second) factor, as it follows from~\eqref{cocycle} and the form of $c$ computed in \cite[\S 3.4]{ER}.

  In turn, the map $c \circ (+_A)$ classifies the extension $(+_A)^* H_{A,B}$, whereas $c_1+c_2$ classifies the extension $p_1^*H_{A,B} + p_2^*H_{A,B}$—the sum is the Baer sum in this case—so that we obtain the isomorphism
  \begin{equation}
    \label{iso-extensions}
    (+_A)^* H_{A,B} \iso p_1^*H_{A,B} + p_2^*H_{A,B}
  \end{equation}
  of central extensions of $A \times A \times B$ by $A\otimes B$. Similarly for the ``variable'' $B$. Furthermore, the commutativity of the diagram alluded to in Remark~\ref{rem:biadditivity} is immediately implied by further pulling back the isomorphism~\eqref{iso-extensions} by $+_B\colon B\times B\to B$, its counterpart for $B$ via $+_A$, and again using~\eqref{cocycle}.
\end{proof}

\subsection{Proof of Theorem \ref{prop1}}
Consider the map $\mu \colon \bbG_m \times \bbG_m \to \cK_2$ obtained using the identification $\bbG_m\simeq \cK_1$ and the multiplication $\cK_1 \times \cK_1 \to \cK_2$. The functor $\cup$ defined as the composite
\[
  \tors_Y(\bbG_m) \times \tors_Y(\bbG_m)
  \xrightarrow{c_{\bbG_m, \bbG_m}}
  \gerb_Y(\bbG_m \otimes \bbG_m) \xrightarrow{\mu_*} \gerb_Y(\cK_2)
\]
is the required bi-additive functor. \qed

The functor $\cup$ is so-named as it categorifies the cup-product (which can be identified with the intersection product
\[
  \H^1(Y, \bbG_m) \times \H^1(Y, \bbG_m) \lto \H^2(Y,\cK_2) \simeq \CH^2(Y) \longleftarrow \CH^1(Y) \times \CH^1(Y)\,.
\]

\begin{remark}
  \label{rem:biext}
  The bi-additivity property of the map $c_{A,B}$ of Proposition~\ref{prop:biadditive} has the following conjectural formal interpretation. The maps $+_A$ and $+_B$, plus the commutative diagram in Remark~\ref{rem:biadditivity} and the proof of Proposition~\ref{prop:biadditive} comprise a structure that can be described as the categorification of a biextension, namely a $\stTors(A\otimes B)$-torsor (hence an $A\otimes B$-gerbe)
  \begin{equation*}
    \stH \lto \stTors(A) \times \stTors(B)
  \end{equation*}
  equipped with partial addition laws $+_A$ (resp.\ $+_B$) giving it the structure of an extension of $\stTors (A)$ (resp.\ $\stTors (B)$) by $\stTors(A \otimes B)$.
\end{remark}

\subsection{Proof of Theorem \ref{prop2}}Our proof will use the results of \S \ref{Leray} on (\ref{lowterm}) for  $\pi:X \to S$ with $A=\cK_2$ on $X$. Proposition \ref{beilinson-lemma} shows that all $\cK_2$-gerbes are horizontal (Definition \ref{horizontal}). The functor 
$\int_{\pi}$ is then defined as the composition of
\[
    \stGerb_X(\cK_2) \xrightarrow{\Theta} \stTors_S(\R^1\pi_*\cK_2) 
    \xrightarrow{\mathrm{Norm}} \stTors_S(\cK_1)\,.
\]
Our first step is to show that $\stGerb_X(\cK_2)'$ is all of   $\stGerb_X(\cK_2)$, in other words, every $\cK_2$-gerbe on $X$ is horizontal. This is proved by showing $\R^2\pi_*\cK_2=0$ which provides the isomorphism 
\[
  E^2_1 \isoto \H^2(X, \cK_2)\, .
\]
We start with the following result, implicit in \cite[A5.1 (iv)]{MR962493}, essentially due to Beilinson-Schechtman. 
\begin{proposition}[Beilinson-Schechtman]\label{beilinson-lemma}
The sheaf $\R^2\pi_*\cK_2$ is zero.
\end{proposition}
This gives a map 
\begin{equation}\label{map-bson}
 \theta \colon  \H^2(X, \cK_2) \lto  \H^1(S, \R^1\pi_*\cK_2)
\end{equation}
using
\[
  \H^2(X, \cK_2)\leftiso E^2_1 \lto \H^1(S, \R^1\pi_*\cK_2)\,.
\]
\begin{proof} Let $N$ be the dimension of $S$, so that $X$ has dimension $N+1$.

For any $s\in S$, we have to show that the stalk of $\R^2\pi_*\cK_2$ at $s$ is zero. By definition, this is the direct limit
\[
  \dirlim_{s\in U}\R^2\pi_*\cK_2(U) = 
  \dirlim_{s\in U}\H^2(\pi^{-1}(U), \cK_2) =
  \dirlim_{s\in U}\CH^2(\pi^{-1}(U))\,,
\]
where the last equality comes from the Bloch-Quillen isomorphism (valid for any smooth variety $V$)
\[
  \H^2(V, \cK_2) \isoto \CH^2(V)\,.
\]
So, we have to show that for any $s \in S$, any open set $U$ containing $s$, and any codimension two cycle $Z$ in $\pi^{-1}(U) \subset X$, there exists an open subset $U' \subset U$ such that the class of $Z$ goes to zero under the map
\[
  \CH^2(\pi^{-1}(U)) \lto \CH^2(\pi^{-1}(U'))\,.
\]
This is clear when $s$ is the generic point $\Spec F(S)$ of $S$: in this case, we take $U'$ to be the complement of $\pi(\abs{Z})$ in $U$. Here we have written $\abs{Z}$ for the support of $Z$. 

The next the case is when $s$ is a point of codimension $i>0$, corresponding to a codimension $i$ subvariety $V$ of $S$. Let us write $Y\subset X$ for $\pi^{-1}(V)$; then $Y$ is a subset of $X$ with codimension $i$. 

For any open $U \subset S$, the condition $s\in U$  means $U \cap V$ is non-empty. Let $U$ be such an open set. There are two cases to consider:
\begin{description}
\item[Case 1] If $\abs{Z}$ is disjoint from $Y$, then we can proceed as before as $\pi(Z)$ is disjoint from $V$, so we take $U'$ to be the complement of $\pi(\abs{Z})$ in $U$. Since $U' \cap V = U \cap V$, we see that $U'\cap V$ is non-empty.

Since $Z$ is in the kernel of the localization sequence for Chow groups 
\[
  \CH^2(\pi^{-1}(U)) \to \CH^2(\pi^{-1}(U) - \abs{Z}) \to 0\,,
\]
it is also in the kernel of the composite map
\[
  \CH^2(\pi^{-1}(U)) \to \CH^2(\pi^{-1}(U) - \abs{Z}) \to \CH^2(\pi^{-1}(U'))\,.
\]
This finishes the proof in this case.
\item[Case 2] If $\abs{Z}$ is not disjoint from $Y$, we can find a codimension two cycle $Z'$ in $\pi^{-1}(U)$ with $[Z] = [Z']\in \CH^2(\pi^{-1}(U))$ which intersects $Y$ transversally. The codimension of the cycle $Z'.Y$ is $i+2$, because its dimension (= maximum of the dimensions of the irreducible components) is $N+1 - i-2 = N-1-i$. Hence the dimension of the image $\pi(Z'.Y)$ is at most $N-1-i$, and so its support $\abs{\pi(Z'.Y)}$ is a proper closed subset of $V = \pi(Y)$. If $U''$ is the complement of $\abs{\pi(Z'.Y)}$ in $U$, then the intersection of $U''$ and $V$ is empty. By definition, the cycle $Z' \cap \pi^{-1}(U'')$ is disjoint from $Y$. This means that the image of $Z'$ (= image of $Z$) under the map
\[
  \CH^2(\pi^{-1}(U)) \lto \CH^2(\pi^{-1}(U''))
\]
is a cycle disjoint from $Y$. By Case 1, we can shrink $U''$ further to $U'$ such that $Z'$ (and hence $Z$ also) is in the kernel of the map
\[
  \CH^2(\pi^{-1}(U'')) \lto \CH^2(\pi^{-1}(U'))\,,
\]
as required.\end{description}
\end{proof}
This gives the functor $\Theta$ appearing in the definition of  
\[ 
    \int_{\pi} \colon \stGerb_X(\cK_2) \xrightarrow{\Theta} \stTors_S(\R^1\pi_*\cK_2) 
    \xrightarrow{\mathrm{Norm}} \stTors_S(\cK_1).
\]
Our next step is the definition of the map $\R^1\pi_*\cK_2 \lto \mathcal{O_S^*}$. 
\begin{remark}
  The same proof shows that if $f\colon Y \to T$ is a smooth proper map of dimension $n$ with $Y$ and $T$ smooth, then $\R^jf_*\cK_j =0$ for all $j>n$. This says that the relative Chow sheaves $\CH^j(Y/T)$ vanish for all $j>n$.
\end{remark}
\subsection{The norm map $\R^1\pi_*\cK_2 \lto \mathcal{O_S^*}$}\label{Norm} This well known map \cite[3.4]{Rost}, \cite[pp.~262-264]{Gillet} arises from the covariant functoriality for proper maps of Rost's cycle modules (Chow groups in our case). We provide the details for the convenience of the reader. Our description proceeds via the Gersten sequence (a flasque resolution of the Zariski sheaf $\cK_2$ on $X$)
\begin{equation}\label{gersten}
    0\lto \cK_2 \lto \eta_*\cK_{2,\eta} \lto \bigoplus_{x \in X^{(1)}} i_*K_1(k(x)) \lto \bigoplus_{y\in X^{(2)}} i_*K_0(k(y)) \to 0\,; 
\end{equation}
here $\eta: \Spec F(X) \to X$ is the generic point of $X$ and $X^{(i)}$ denotes the set of points of codimension $i$ of $X$. 
For any $U$ open in $S$, the norm map
\begin{equation*}
    \H^1(\pi^{-1}(U), \cK_2) \lto \mathcal O^*_S(U)
\end{equation*}
is obtained as follows. Since the first group is the homology at degree one of (\ref{gersten}), we proceed by constructing a map
\[
  \bigoplus_{x \in \pi^{-1}(U)^{(1)}} i_*K_1(k(x)) \to \mathcal O^*_S(U)\,.
\]
For each such $x \in \pi^{-1}(U)$ of codimension one, the map $x \to \pi(x)$ is either finite or not, and it is zero in the second case. In the first case, there is a norm map
\[
  k(x)^* \to k(\pi(x))^*\,;
\]
since $x$ has codimension one in $X$, its image $\pi(x)$ is the generic point of $S$ and hence the above norm map is a map
\[
  k(x)^* \to F(S)^*\,.
\]
An element of $\H^1(\pi^{-1}(U), \cK_2)$ arises from a finite collection of functions $f_x\in k(x)^*$ (for $x\in \pi^{-1}(U)$ of codimension one which is finite onto its image) which is in the kernel of the map
\[
  \bigoplus_{x \in \pi^{-1}(U)^{(1)}} i_*K_1(k(x)) \to \bigoplus_{y \in \pi^{-1}(U)^{(2)}} i_*K_0(k(y))\,.
\]
On each component, this is the ord or valuation map. One checks that this means that the (finite) product of the norms of $f_x$ is an element of $F(S)^*$ with no poles on $U$ and hence defines an element of $\mathcal O_S^*(U)$.  This gives the required functor
\[
    \stTors_S(\R^1\pi_*\cK_2) \xrightarrow{\mathrm{Norm}} \stTors_S(\cK_1)\,,
\] 
completing the definition of the functor $\int_{\pi}$ of Theorem \ref{prop2}.

\section{Comparison with Deligne's construction}
\label{comp-deligne}

Given line bundles $L$ and $M$ (viewed as $\bbG_m$-torsors) on $X$, consider the $\cK_2$-gerbe $G_{L,M}$ on $X$. By Proposition \ref{beilinson-lemma}, the element $[G_{L,M}]$ of $\H^2(X, \cK_2)$ actually lives in $E_1^2$ and hence $G_{L,M}$ is horizontal. By Lemma \ref{lem:exercise}, $\Theta(G_{L,M})$ is a $\R^1\pi_*\cK_2$-torsor. By definition, $\int_{\pi}G_{L,M}$ is its pushforward along the norm map of \S \ref{Norm},
\[
 \mathrm{Norm} \colon \R^1\pi_*\cK_2 \lto \mathcal{O_S^*}\,,
\]
which gives a line bundle $(L,M)$ on $S$. In this section, we show that this gives Deligne's line bundle $\< {L,M}\>$. Since  $(L,M)$ is bi-additive and its construction is functorial, this reduces to showing the identity in Theorem \ref{prop3}:
\[
    \< {\mathcal{O}(D), \mathcal{O}(E)} \iso ( {\mathcal{O}(D), \mathcal{O}(E)})\,,
\]
for any relative Cartier divisors $D$ and $E$ on $X$ with $D$ effective.

\subsection{Comparison}

To show that $\< {\mathcal{O}(D), \mathcal{O}(E)}$ is isomorphic to $( {\mathcal{O}(D), \mathcal{O}(E)})$, one just has to show that they are equal in  $\H^1(S, \mathcal{O}^*)$. This amounts to showing that the diagram below is commutative:
 \begin{equation}\label{diagram}
    \begin{tikzcd}
    & \H^1(X, \mathcal{O}^*) \ar[dl,"\eta"'] \ar[dr,"\cup"] \\
      \H^1(D, \mathcal{O}^*) \ar[rr,"\lambda"] \ar[d,"N_{D/S}"']  && \H^2(X, \cK_2) \ar[d,"\theta"] \\
      \H^1(S, \mathcal{O}^*)  && \H^1(S, \R^1\pi_*\cK_2) \ar[ll,"\mathrm{Norm}"']
    \end{tikzcd}
  \end{equation}
The map $\eta$ is the restriction to $D$ of a line bundle $\mathcal{O}(E)$. The map $\cup$ sends $\mathcal{O}(E)$ to its cup-product with $\mathcal{O}(D)$. The boundary map $\lambda$ in the localization sequence
\[0
  \to \cK_{2,X} \lto j_*\cK_{2,U} \lto i_*\cK_{1,D} \to 0
\]
for $X$, $U= X - D$, and $D$. The map $\theta$ is the map (\ref{map-bson})
\[
  \H^2(X, \cK_2) \lto  \H^1(S, \R^1\pi_*\cK_2)\,.
\]
The commutativity of (\ref{diagram}) is an implicit consequence of the axiomatics of Rost \cite{Rost}, but we provide a direct proof. 

\subsubsection{The top triangle of (\ref{diagram})} We first prove the commutativity of the top triangle of (\ref{diagram}). Let $\{U_i\}$ be a Zariski open cover of $X$ such that $D$ and $E$ are principal divisors on $U_i$. Let $\{f_i\}$ be defining equations for $D$ and $\{g_i\}$ be defining equations for $E$. Then, 
\[
  \{a_{ij}:= \frac{f_i}{f_j} \in \mathcal O^*(U_i \times U_j)\}, \quad
  \{b_{ij}:= \frac{g_i}{g_j} \in \mathcal O^*(U_i \times U_j)\}
\]
are cocycle representatives for  $\mathcal O(D)$ and $\mathcal O(E)$. By the explicit description \cite[(1-18)]{MR2362847} of the cup-product map in \v{C}ech cohomology, the map $\cup$ sends $\{b_{ij}\}$ to the $2$-cocycle
\begin{equation}
  \label{cup-product}
  \{(a_{ij}, b_{jk})\}\in K_2(U_i \times U_j \times U_k)\,.
\end{equation}

Given a cocycle $s_{ij}\in \mathcal O^*(U_i \times U_j \times D)$ relative to the cover $\{U_i \times D\}$ of $D$, one computes its image under $\lambda$ as follows. Pick $\tilde{s}_{ij} \in K_2 (U_i \times U_j \times U)$ whose tame symbol along $D$ is $s_{ij}$; then check that its \v{C}ech boundary $\partial( \tilde{s}_{ij})$ (a $2$-cochain with values in $\cK_{2,U}$) is zero when viewed as a cochain with values in $i_*\cK_{1,D}$. This means that $\partial (\tilde{s}_{ij})$ is a $2$-cocycle with values in $\cK_{2,X}$; this is defined to be the image of $s_{ij}$ under $\lambda$. Let us apply this to compute the image of $\mathcal O(E)\vert_D$ under $\lambda$.

Let $\bar{b}_{ij}$ be the image of $b_{ij}$ under the map
\[
  \mathcal O^*(U_i \times U_j) \to \mathcal O^*(U_i \times U_j \times D)\,.
\]
The cocycle $\{\bar{b}_{ij}\}$ represents $\mathcal O(E)\vert_D$. To compute its image under $\lambda$, consider the element (symbol)
\[
  t_{ij} = (f_i,b_{ij}) \in K_2(U \times U_i \times U_j).
\]
We know that $b_{ij}$ is a unit in $U_i \times U_j$ and so defines an element of $K_1(U_i \times U_j)$; we know $f_i$ is the defining equation of $D$ on $U_i$ and so it is a unit on $U \times U_i$ and thus $f_i$ defines an element of $K_1(U \times U_i)$. So $t_{ij}$ is a well-defined element of $K_2(U \times U_i \times U_j)$. If $v$ denotes the valuation
\[
  F(X)^* \lto \mathbb{Z}
\]
defined by the divisor $D$, the tame symbol map is the map
\[
  K_2(U) \to K_1(D), \quad (a,b) \mapsto (-1)^{v(a)v(b)}.~ \overline{\bigg(\frac{a^{v(b)}}{b^{v(a)}}\bigg)}\,.
\]
Since $v(f_i) =1$ and $v(b_{ij}) =0$, we see that $t_{ij}$ maps to the element 
\[
  (-1)^{1 \times 0}.~\overline{\bigg(\frac{f_i^0}{b_{ij}^1}\bigg)} = \bar{b}_{ij}^{-1}\,.
\]
So the cochain $\{t_{ij}\}$ lifts the inverse of the cocycle $\{\bar{b}_{ij}\}$. Its \v{C}ech boundary (which represents the image under $\lambda$ of the inverse of $\{\bar{b}_{ij}\}$)
\[
  t_{ij} - t_{ik} + t_{jk} = (f_i,b_{ij}) - (f_i,b_{ik}) + (f_j,b_{jk})
\]
is a $2$-cocycle with values in $\cK_2$. Since 
\[
  \biggl\{b_{ij} = \frac{g_i}{g_j}\biggr\}
\]
is a cocycle, the relation 
\[
  b_{ik} = b_{ij}+ b_{jk}
\]
holds. Using this, the image of the inverse of $\{\bar{b}_{ij}\}$ under $\lambda$ is given by the negative of the element in (\ref{cup-product}):
\[
  (f_i, b_{ij}) - (f_i, b_{ij}) -(f_i, b_{jk}) +(f_j, b_{jk})
  = (\frac{f_j}{f_i}, b_{jk}) = -(a_{ij}, b_{jk})\,.
\]
This says that $\lambda$ maps $\{\bar{b}_{ij}\}$ to the class of the cup product of $\mathcal O(D)$ and $\mathcal O(E)$ in $\H^2(X ,\cK_2)$ thus completing the proof of the commutativity of the top triangle in (\ref{diagram}).

\subsubsection{The bottom square of (\ref{diagram})}
We begin with an explicit description of the map 
\[
  \theta \colon \H^2(X, \cK_2) \to \H^1(S, \R^1 \pi_*\cK_2)
\]
in (\ref{map-bson}).

Let $G$ be a $\cK_2$-gerbe on $X$. As $\CH^2(X)= \H^2(X,\cK_2)$ (Bloch-Quillen), we can pick a codimension-two cycle $c$ representing $[G]$ on $X$. As $G$ is horizontal, there exists an open cover $\{V_{\alpha}\}$ of $S$ such that $[G] =0\in \H^2(W_{\alpha}, \cK_2)$, with $W_{\alpha} = \pi^{-1}(V_{\alpha})$; note $\{W_{\alpha}\}$ is an open cover of $X$. In terms of the Gersten complex
\[
  0\lto \cK_2 \lto \eta_*\cK_{2,\eta} \lto
  \bigoplus_{x \in W_{\alpha}^{(1)}} K_1(k(x)) \xrightarrow{\mathrm{ord}}
  \bigoplus_{y\in W_{\alpha}^{(2)}} K_0(k(y)) \to 0\,, 
\]
which computes the cohomology of $\cK_2$ on $W_{\alpha}$, we have that the vanishing in $H^2(W_{\alpha}, \cK_2)$ of the  restriction $c_{\alpha}$ of the codimension-two cycle $c$ representing $[G]$ on $W_{\alpha}$. Then, there exists an element $h_{\alpha} \in \bigoplus_{x \in W_{\alpha}^{(1)}}~K_1(k(x))$ such that $\textrm{ord}(h_{\alpha}) =c_{\alpha}$ in the sequence on $W_{\alpha}$. So $h_{\alpha}$ is a collection of divisors in $W_{\alpha}$ whose associated functions cut out together the codimension-two cycle $c$. Since $\textrm{ord}(h_{\alpha}) = \textrm{ord}(h_{\alpha'})$ on $W_{\alpha} \cap W_{\alpha'}$, we see that the element 
$r_{\alpha,\alpha'}\coloneq h_{\alpha} -h_{\alpha'}$  on $W_{\alpha} \cap W_{\alpha'}$ defines an element of $\H^1(W_{\alpha} \cap W_{\alpha'}, \cK_2)$. The cocycle condition is a formal consequence:
\[
  r_{\alpha,\alpha'} + r_{\alpha',\alpha''} + r_{\alpha'',\alpha} =0\,.
\]
Namely, $\{r_{\alpha,\alpha'}\}$ defines a Čech $1$-cocycle on $S$ with values in $\R^1 \pi_*\cK_2$; this is the element $\theta(G)$. Taking norms down to $S$ gives a Čech $1$-cocycle
\[ 
    \tilde{r}_{\alpha,\alpha'} \coloneq N_{D/S}\biggl( \frac{h_{\alpha}}{h_{\alpha'}} \biggr)
\] 
with values in $\mathbb G_m$ on $S$. This completes the description of the maps in the bottom square of (\ref{diagram}).

With all this in place, it is now easy to show that the bottom square of (\ref{diagram}) commutes. Recall the defining equations $g_i$ of $E$ relative to the open cover $\{U_i\}$ of $X$. Restricting $\mathcal O(E)$ to $D$ and applying $\lambda$ gives the gerbe  $G= G_{\mathcal O(D), \mathcal O(E)}$, by the commutativity of the top triangle of (\ref{diagram}). We use the above description to compute the image of the gerbe under $\theta$; we see that $h_{\alpha}$ can be taken to be the collection of functions $\bar{g}_{i, \alpha} = \bar{g}_i\lvert_{D_{\alpha, i}}$ on $D_{\alpha, i}= D \cap W_{\alpha}\cap U_i$ which cuts out the codimension-two cycle corresponding to the intersection of $D$ and $E$. The norm down to $S$ of the corresponding $r_{\alpha, \alpha'}$ gives the Čech-cocycle with values in $\cK_1$ of $S$; this is the image of $\mathcal O(E)$ along one part of the bottom square in (\ref{diagram}).   

On the other hand, consider the image of $\mathcal O(E)$ under the left vertical map of (\ref{diagram}). Let
\[
    \bar{g}_{i, \alpha} = \bar{g}_i\lvert_{D_{\alpha, i}}, \quad 
    e_{\alpha}= \prod_{i}N_{D_{\alpha, i}/{S}} \bigl(\bar{g}_{i, \alpha}\bigr) \,.
\]
The image of $\mathcal O(E)$ under the map $N_{D/S}$ is given by the cocycle
\[
    c_{\alpha, \alpha'}\coloneq \frac{e_{\alpha}}{e_{\alpha'}}\in 
    \cO^*(V_{\alpha}\cap V_{\alpha'})\,.
\]
It is clear that $c_{\alpha, \alpha'}$ is equal to $\mathrm{Norm}(r_{\alpha, \alpha'})$. This shows the commutativity of the diagram (\ref{diagram}), since the image of $\cO(E)$ along the vertical left map of (\ref{diagram}) gives Deligne's line bundle $\<{\cO(D), \cO(E)}\>$, and the image along the other side of (\ref{diagram}) is
\[(
    \mathcal O(D), \mathcal O(E)) = 
    \mathrm{Norm}\circ \Theta\circ\cup (\cO(D), \cO(E)) = 
    \mathrm{Norm}\circ \Theta\circ (G_{\cO(D), \cO(E)}) = 
    \int_{\pi} G_{\cO(D), \cO(E)}.
\]
This proves Theorem \ref{prop3} and therefore Theorem \ref{Main}.

\section{Picard stacks and their endomorphisms}
\label{sec:picard-stacks-endom}

Here and elsewhere in this paper ``Picard stack,'' or ``Picard category,'' means ``strictly commutative Picard'' in the sense of Deligne \cite{10.1007/BFb0070724}. Namely, if we denote the monoidal operation of $\stP$ simply by $\mathnormal{+} \colon \stP \times \stP \to \stP$, then the symmetry condition given by the natural isomorphisms $\sigma_{x,y} \colon x + y \to y + x$ must satisfy the additional condition that $\sigma_{x,x} = \id_{x+x}$. Such stacks have the pleasant property that there exists a two-term complex of abelian sheaves $d\colon A^{-1} \to A^0$ such that $\stF$ is equivalent, as a Picard stack, to the one associated to the action groupoid formed from the complex. We denote this situation by
\begin{equation*}
  \stP \simeq \bigl[ A^{-1} \stackrel{d}{\lto} A^0 \bigr]\sptilde\,.
\end{equation*}
A classical example arises from the well known divisor exact sequence of Zariski sheaves
\begin{equation}\label{divisorsequence}
 0 \to \bbG_m \to \eta_*F(Y)^{\times} \to \bigoplus_{y\in Y^1} (i_y)_*\mathbb Z \to 0\,,
\end{equation}
where $Y$ is a smooth scheme over a field $F$ and the sum is over the set $Y^1$ of points of codimension one in $Y$. We get the equivalence:
\begin{equation*}
  \stTors_Y(\bbG_m) \simeq
  \bigl[ \eta_*F(Y)^{\times} \to \bigoplus_{y\in Y^1} (i_y)_*\mathbb Z \bigr]\sptilde
\end{equation*}
In the sequel we shall denote by $\stCH^1_Y$ the Picard stack on the right hand side of the above relation and by $\CHcat^1(Y)$ the Picard category of its global sections. Therefore we have $\tors_Y(\bbG_m) \simeq \CHcat^1(Y)$.

Still from \cite{10.1007/BFb0070724}, we have that morphisms and natural transformations form a Picard stack $\stHom(\stP,\stQ)$, where the additive structure is defined pointwise: if $F,G$ are two objects, then $(F+G)(x) \coloneq F(x) +_{\stQ} G(x)$, for any object $x$ of $\stP$. It is immediate to verify that this is symmetric and in fact strictly commutative if $+_{\stQ}$ is.

\subsection{Ring structures}
\label{sec:ring-structures}

We set $\stEnd (\stP) = \stHom (\stP,\stP)$. By the above considerations, it is a Picard stack, but the composition of morphisms gives it an additional unital monoidal structure, with respect to which $\stEnd (\stP)$ acquires the structure of a stack of ring groupoids—also known as categorical rings—of the sort described in \cite{rings-tac2015} (see also \cite{drinfeld2021notion} for a résumé). Note that the ``multiplication'' monoidal structure, being given by composition of functors, is strictly associative.
\begin{remark}
  \label{rem:butterflies}
  If $\stP \simeq [A^{-1}\to A^0]\sptilde$, then by \cite{ButterfliesI} $\stEnd(\stP)$ is equivalent, as a stack of ring groupoids, to $\stCorr(A^\bullet,A^\bullet)$, the stack whose objects are butterfly diagrams: an object is given by an extension $0 \to A^\bullet \to E \to A^0\to 0$ such that its pullback to $A^{-1}$ via $d \colon A^{-1}\to A^0$ is trivial. Morphisms are morphisms of extensions. $\stCorr(A^\bullet,A^\bullet)$ is a stack of ring groupoids: the ``$+$'' is given by the Baer sum of extensions; the ``$\times$'' is given by concatenation of butterflies described in \loccit This structure is associative, but not strictly so.
\end{remark}

\subsection{Quotients and colimits}
\label{sec:quotients-colimits}

Let $F\colon \stP\to \stQ$ be a morphism of Picard stacks. Its cokernel $\Coker F$ is the stack associated to the following construction, the details of which can be found in the literature (see, \eg \cite{VITALE2002383,kv2000}).\footnote{This is valid for Picard categories and stacks that not necessarily strictly commutative.} Let assume $F \colon \mathbf{P}\to \mathbf{Q}$ is a morphism of Picard \emph{categories.} The cokernel $\Coker F$ is a Picard category defined as follows:
\begin{enumerate}
\item its class of objects is the same as that of $\mathbf{Q}$;
\item a morphism $[f , a] \colon x \to y$ is an equivalence class of pairs $(f, a)$, where $f$ is morphism $f \colon x \to x + F(a)$ in $\mathbf{Q}$,  $a \in \Obj\mathbf{P}$, and two pairs $(f,a)$ and $(g,b)$ are equivalent if there exists an arrow $u \colon a\to b$ in $\mathbf{P}$ and the diagram
  \begin{equation*}
    \begin{tikzcd}[sep=small,cramped]
      & x \ar[dl,"f"'] \ar[dr,"g"] \\
      y + Fa) \ar[rr,"y+F(u)"'] && y + F(b)
    \end{tikzcd}
  \end{equation*}
  commutes.
\item The monoidal structure is defined to be that of $\mathsf{Q}$ on objects and by the class of the composite arrow
  \begin{equation*}
    \begin{split}
      x + y \to (x + F(a)) + (y + F(b)) &\isoto (x + F(a)) + (y + F(b)) \\
      &\isoto (x + y) + (F(a) + F(b)) \isoto (x + y) + F(a + b)
    \end{split}
  \end{equation*}
  if $[f,a]\colon x\to x'$ and $[g,b]\colon y \to y'$.
\end{enumerate}
There is a canonical functor $p_F\colon \mathbf{Q}\to \Coker F$ which is the identity on objects and sends an arrow $f\colon x\to y$ in $\mathbf{Q}$ to the class of the composite:
\begin{equation*}
  \begin{tikzcd}[sep=small,cramped]
    x \ar[r,"f"] & y \ar[r,"\simeq"] &
    y + 0_{\mathsf{Q}} \ar[r,"\simeq"] &
    y + F (0_{\mathsf{P}})
  \end{tikzcd}\,.
\end{equation*}
There is an isomorphism $\pi_F\colon p_F\circ F \Rightarrow 0\colon \mathbf{P} \to \mathbf{Q}$ given by
\begin{math}
  \pi_{F,a} \colon
  \begin{tikzcd}[cramped,sep=small]
    F(a) \ar[r,"\simeq"] & 0_{\mathbf{Q}} + F(a)
  \end{tikzcd}\,.
\end{math}
It follows that we have an abelian group isomorphism
\begin{equation*}
  \pi_0(\Coker F) \iso
  \Coker \bigl(\pi_0(F) \colon \pi_0(\mathbf{P}) \to \pi_0(\mathbf{Q})\bigr)\,.
\end{equation*}
As mentioned, if $F\colon \stP\to \stQ$ is a morphism of Picard \emph{stacks,} we define $\Coker F$ to be the Picard stack associated to the pseudo-functor
\begin{equation*}
  U \rightsquigarrow \Coker \bigl(F_U \colon \stP (U) \to \stQ(U) \bigr)
\end{equation*}
where $U$ is in the base site.

In the following section, we apply this construction to the diagram
\begin{equation*}
  \begin{tikzcd}[cramped]
    \stP_1 \ar[r,"F_1"] & \stQ \ar[r,leftarrow,"F_2"] & \stP_2
  \end{tikzcd}
\end{equation*}
and the resulting morphism $F_1+F_2\colon \stP_1\times \stP_2\to \stQ$, defined on objects by $(F_1 + F_2) (x_1 , x_2) = F_1(x_1) + F_2(x_2)$. We shorten or notation and simply write $\Coker (F_1+F_2)$ as $\stQ/(\stP_1+\stP_2)$. By the above recollection we have
\begin{equation*}
  \pi_0\bigl( \stQ/(\stP_1+\stP_2) \bigr) \iso
  \pi_0(\stQ) / (\pi_0(\stP_1) + \pi_0(\stP_2))\,.
\end{equation*}

\section{Categorification of correspondences}
\label{sec:categ-corr}

In this section, $S=\Spec~F$ and a curve $C$ is a smooth projective connected one-dimensional scheme over $S$. For simplicity, we assume (just in this section) that $F$ is algebraically closed.

The main result (Theorem \ref{nuovo}) of this section is an application of Theorem \ref{Main} using \S \ref{sec:picard-stacks-endom} to the categorification of well known identities (\ref{ek}) and (\ref{caar}) about correspondences on the self-product of a curve. We will work with the category $\cV$ whose objects are curves and the maps are (correspondences) $\Hom_{\cV}(D, C) = \Div(C \times D)$ with composition defined by product of correspondences.

In the following we are stating our results for Picard categories, but there are parallel statements for the corresponding Picard stacks.

\subsection{Categorifying $\CH^1(Y)$} For any smooth scheme $Y$ over $S$, it follows from (\ref{divisorsequence}) that the Chow group $\CH^1(Y)$ is isomorphic to the Picard group $\Pic(Y) = \H^1(Y, \bbG_m)$ of $Y$; the Picard category $\CHcat^1(Y)$ 
is canonically equivalent to the Picard category of $\mathbb G_m$-torsors or line bundles on $Y$.

The Picard category $\CHcat^1(Y)$  categorifies the Chow group $\CH^1(Y)$ of divisors:
\begin{enumerate}
\item $\CH^1(Y) = \Pic(Y) =\H^1(Y, \bbG_m) = \pi_0(\CHcat^1(Y))$;
\item Any map $f \colon Y \to Y'$ of smooth schemes defines an additive functor of Picard categories
  \begin{equation*}
    f^*: \CHcat^1(Y') \to \CHcat^1(Y), \qquad L \mapsto f^*L\,;
  \end{equation*}
  the induced map on $\pi_0$ is the pullback of divisors $f^*: \CH^1(Y') \to \CH^1(Y)$. 
\end{enumerate}

For a curve $C$, let $\Pic^0(C)$ be the kernel of the degree map $\Pic(C) \to \mathbb Z$. If $\CHcat^1(C)^0$ is the sub-Picard category of $\CHcat^1(C)$ consisting of line bundles of degree zero, then $\pi_0(\CHcat^1(C)^0) = \Pic^0(C)$.  

\subsection{Correspondences} We refer to \cite[Chapter 16]{MR1644323} for details. Let $C$ and $D$ be curves and let $\pi_C$ and $\pi_D$ be the two projections on $C\times D$. A correspondence $\alpha: D \vdash C$ from $D$ to $C$ is a divisor $\alpha$ on $C \times D$. It defines a line bundle $\cO(\alpha)$ on $C \times D$.  The correspondence $\alpha$ acts on divisors: it induces a map
\begin{equation*}
  \alpha^* \colon \Pic(D) \to \Pic(C) \quad m \mapsto  (\pi_C)_*(\pi_D^*m.\alpha)
\end{equation*}
which sends a divisor $m$ on $D$ to the pushforward along $\pi_C$ of the intersection of $\alpha$ and $\pi_D^*m$ on $C\times D$. It restricts to a map $\textrm{Pic}^0(D) \to \textrm{Pic}^0(C)$: if $m$ has degree zero, then so does $\alpha^*(m)$. We get a homomorphism
\begin{equation}
  \label{ek}
  T\colon \CH^1(C \times D) \to \Hom(\Pic(D), \Pic(C)) \to \Hom(\Pic^0(D), \Pic^0(C))\,,
  \quad \alpha \mapsto \alpha^*  
\end{equation}
as $\alpha^*$ depends only on the class of $\alpha$ in $\CH^1(C \times D)$.

Degenerate correspondences \cite[Example 16.1.2]{MR1644323} constitute the subgroup $I(D,C) =\pi_C^*(CH^1(C)) + \pi_D^*(CH^1(D))$ of $\CH^1(C \times D)$. The map $T$ induces an isomorphism \cite[Proposition 3.3, Theorem 3.9]{MR1265529}
\begin{equation}
  \label{doh}
  T\colon \frac{\CH^1(C \times D)}{I(D,C)} \to \Hom(\Pic^0(D), \Pic^0(C));
\end{equation} 
 see \cite[Chapter 11, Theorem 5.1]{MR2062673} for another proof when $F=\mathbb C$. Over  a non-algebraically closed field, the isomorphism holds if $C$ and $D$ have rational points. 

Composition of correspondences induces a ring structure on $\CH^1(C \times C)$ with $I(C,C)$ as an ideal \cite[Example 16.1.2]{MR1644323}.  It is known that 
\begin{itemize}
\item \cite[Corollary 16.1.2]{MR1644323} the map
  \begin{equation}\label{teen}
    T: \CH^1(C \times C) \to \End(\CH^1(C))\,, \quad \alpha \mapsto \alpha^*
  \end{equation}
  is a homomorphism of rings.
\item $T$ induces a ring isomorphism \cite[Example 16.1.2(c)]{MR1644323}
  \begin{equation}
    \label{caar}
    T\colon \frac{CH^1(C \times C)}{I(C,C)} \to \End(\Pic^0(C)),    
  \end{equation}
as  $F$ is algebraically closed; see \cite[Chapter 11, Theorem 5.1]{MR2062673} for a proof when $F=\mathbb C$. 
\end{itemize}

The following result provides a categorification of the above statements.
\begin{theorem}
  \label{nuovo}
  There is an additive functor of Picard categories
    \begin{equation}
      \tilde{T}\colon \CHcat^1(C \times D) \to
      \stHom(\CHcat^1(D), \CHcat^1(C)) \to
      \stHom(\CHcat^1(D)^0, \CHcat^1(C)^0)   
    \end{equation}
    which induces (\ref{ek}) on $\pi_0$. $\Tilde{T}$ has the following properties:
 \begin{enumerate}[label=(\roman*)]
  \item Let $M(D,C)= \Coker(\pi_C^*+ \pi_D^*)$ be the cokernel of the additive functors
    \begin{equation}
      \pi_C^*\colon \CHcat^1(C) \to \CHcat^1(C \times D)
      \leftarrow \CHcat^1(D) \reflectbox{$\colon$} \pi_D^*.
    \end{equation} 
     Then $\tilde{T}$ induces an additive functor
    \begin{equation}
      \tilde{T}_{D,C} \colon M(D,C) \to \stHom(\CHcat^1(D)^0, \CHcat^1(C)^0)
    \end{equation}
    which, on $\pi_0$, is (\ref{doh}).
  \item $\CHcat^1(C \times C)$ comes naturally equipped with the structure of a  categorical ring (\S \ref{sec:ring-structures}) which, on $\pi_0$, is the composition of correspondences.
  \item the functors 
    \begin{equation}
      \Tilde{T} \colon \CHcat^1(C \times C)\to \stEnd(\CHcat^1(C)), \quad
      \Tilde{T}_{C,C} \colon M(C,C) \to \stEnd(\CHcat^1(C)^0)
    \end{equation}
    are functors of categorical rings (\S \ref{sec:ring-structures}). These induce (\ref{teen}), (\ref{caar}) on $\pi_0$.
  \end{enumerate}
\end{theorem}
\begin{remark}
  The assignment $C \mapsto \CHcat^1(C)$ and $g\in \Div (C\times D) \mapsto \Tilde{T}(\cO (g))$ comprise a pseudo-functor from $\cV$ (the category of curves and correspondences) to the 2-category of Picard categories (or stacks).
\end{remark}

\subsection{Proof of Theorem \ref{nuovo}} The existence of $\tilde{T}$ is provided by the following lemma.

\begin{lemma} For any correspondence $\alpha:D \vdash C$, the map $\alpha^*: \CH^1(D) \to \CH^1(C)$ is induced by an additive functor $\tilde{\alpha}^* \colon \CHcat^1(D) \to \CHcat^1(C)$. This functor restricts to a functor $\CHcat^1(D)^0 \to \CHcat^1(C)^0$. Further, if $\beta$ is another correspondence, then $\widetilde{(\alpha+\beta)}^* = \tilde{\alpha}^*+ \tilde{\beta}^*$ as additive functors.
\end{lemma}
\begin{proof} This is a simple application of Theorem \ref{Main}. Given a line bundle $M$ on $D$, consider the pair $\pi_D^*M$ and $\cO (\alpha)$ of line bundles on $C \times D$; \cite{ER} constructs a $\cK_2$-gerbe $G_{(\cO (\alpha), \pi_D^*M)}$on $C \times D$. As this gerbe  is horizontal by Proposition \ref{beilinson-lemma}, one can integrate it along $\pi_C:C \times D \to C$ to get a line bundle on $C$:
  \begin{equation*}
    \tilde{\alpha}^*M = \int_{\pi_C}~G_{(\mathcal O(\alpha), \pi_D^*M)} \,.
  \end{equation*}
 Both the additivity of $\tilde{\alpha}^*$ and the property $\widetilde{(\alpha+\beta)}^* = \tilde{\alpha}^*+ \tilde{\beta}^*$ follow from the bi-additivity (Theorem \ref{prop2}) of $G$. If the line bundle $M$ on $D$ has degree zero, then so does the line bundle $ \tilde{\alpha}^*M$ on $C$ as its class in $CH^1(C)$ is $\alpha^*(M)$ which has degree zero. 
\end{proof}
 This gives us a bi-additive functor of Picard categories
\[\CHcat^1(C \times D) \times \CHcat^1(D) \to \CHcat^1(C), \qquad (\alpha,M) \mapsto \tilde{\alpha}^*M,\] and an additive functor 
\[\tilde{T}: \CHcat^1(C \times D) \to \stHom(\CHcat^1(D), \CHcat^1(C)) \to \stHom(\CHcat^1(D)^0, \CHcat^1(C)^0)\]
where, for any pair $P$, $P'$ of Picard categories,  $\stHom(P,P')$ is the Picard category of additive functors from $P$ to $P'$. 

 Statement (i) of Theorem \ref{nuovo} concerns the factorization of $\tilde{T}$ as
\begin{equation}\label{eq2}
 \tilde{T}:M(D,C) \to \stHom(\CHcat^1(D)^0, \CHcat^1(C)^0)   
\end{equation} 
This, in turn, follows from the triviality of $\tilde{T}$ on $\pi_C^*\CHcat^1(C)$ and $\pi_D^*\CHcat^1(D)$: 
\begin{itemize}
  \item {\bf $\tilde{T}$ restricted to $\pi_D^*\CHcat^1(D)$.} 
    
    If $g:D \vdash C$ is the pullback $\pi_D^*x$ of a divisor $x$ on $D$, then $\tilde{T}(g)$ applied to a line bundle $L$ on $D$ is defined as $\int_{\pi_C}G_{(\mathcal O(g), \pi_D^*L)}$. As the construction of the $\cK_2$-gerbe is functorial, we have
    \[G_{(\pi_D^*x, \pi_D^*L)}= \pi_D^*G_{(x,L)};\] as $H^2(D, \cK_2) =0$, the $\cK_2$-gerbe $G_{(x,L)}$ on $D$ is trivializable. Since $\int_{\pi_C}$ is an additive functor, $\tilde{T}(g)(L)$ is trivializable. It follows that
     \[\tilde{T}(g): \CHcat^1(D)^0 \to \CHcat^1(C)^0\] is the trivial functor. 
   
    \item {\bf $\tilde{T}$ restricted to $\pi_C^*\CHcat^1(C)$.}
    
    If $g:D \vdash C$ is $\pi^*_Cx$ of a divisor $x$ on $C$ and $m=\sum m
    _j y_j$ is a divisor on $D$, then $\tilde{T}(g)(m)$ corresponds to the ${\mathrm deg}~m$-th power of the line bundle $\mathcal O(x)$ and hence is trivial when $m$ has degree zero. This can be seen as follows: $\tilde{T}(g)(m)$ is the object corresponding to the line bundle
    \[\<{\pi_D^*m,\pi^*_C x} = \otimes_j \<{\pi^*_D y_j, \pi_C^* x} ^{\otimes m_j}.\]
    Since $\pi_C: C \times y_j \hookrightarrow C \times D \to C$ is an isomorphism for any closed point $y_j$ of $D$, one has $\<{\pi^*_D y_j, \pi_C^* x} = \mathcal O(x)$ by (\ref{norm-finite}). By bi-additivity, $$\<{\pi_D^*m,\pi^*_C x}= (\mathcal O(x))^{\mathrm{deg}~m}.$$
    If $m$ has degree zero, then $\tilde{T}(g)(m)$ is trivializable.

So the functor 
\[\tilde{g}^*~=~\tilde{T}(g): \CHcat^1(D)^0 \to \CHcat^1(C)^0\]
is trivial.
\end{itemize}
This completes the proof of (i) of Theorem \ref{nuovo}.

\subsection{Composition}We show that $\tilde{T}$ is compatible with composition of correspondences. Let $X =C_1 \times C_2 \times C_3$ be the product of three curves $C_1, C_2, C_3$ and let $\pi_{ij}:X\to C_i \times C_j$ be the projections.  If $g:C_2\vdash C_1$ is a correspondence on $C_1 \times C_2$ and $h:C_3\vdash C_2$ on $C_2 \times C_3$, one can compose $g$ and $h$ to get a correspondence $g\circ h:C_3\vdash C_1$ on $C_1 \times C_3$:
\[g\circ h = (\pi_{13})
_* (\pi_{23}^*h~.~\pi_{12}^*g),\] 
by pulling back $g$ and $h$ to $X$ and intersecting them and pushing forward via $\pi_{13}$ to $C_1 \times C_3$. This gives a bi-additive map
\[\circ: \CH^1(C_1 \times C_2) \times \CH^1(C_2 \times C_3) \to \CH^1(C_1 \times C_3).\]
\begin{lemma}\label{l1}
The above bi-additive map is induced by a bi-additive functor
\[\tilde{\circ}: \CHcat^1(C_1 \times C_2) \times \CHcat^1(C_2 \times C_3) \to \CHcat^1(C_1 \times C_3).\]
\end{lemma}
\begin{proof}
 The functor is defined as follows: The pair $\pi_{12}^*\cO(g)$ and $\pi_{23}^* \cO(h)$ of line bundles on $X$ give rise to a $\cK_2$-gerbe $G_{(\pi_{12}^*\cO(g),\pi_{23}^* \cO(h))}$ on $X$. Since it is horizontal (Proposition \ref{beilinson-lemma}) for the map (a relative curve)  $\pi_{13}:X \to C_1 \times C_3$, we can integrate it along  $\pi_{13}$ to obtain a line bundle $\<{\pi_{12}^*\cO(g),\pi_{23}^* \cO(h)}$ on $C_1 \times C_3$. The functor $\tilde{\circ}$, in the notation of Theorem \ref{Main}, is 
 \[\tilde{\circ}:(g,h) \mapsto \int_{\pi_{13}}G_{(\pi_{12}^*\cO(g),\pi_{23}^* \cO(h))} = \<{\pi_{12}^*\cO(g),\pi_{23}^* \cO(h)}.\]
 It follows from Theorem \ref{Main} that $\tilde{\circ}$ induces $\circ$ on $\pi_0$.
 \end{proof}
 Taking $C_1=C_2=C_3=C$ proves (ii) of Theorem \ref{nuovo}. 
 
\begin{lemma}\label{compos-eh?}The functor $\tilde{T}$ is compatible with composition: namely, the diagram
 \begin{equation}
    \label{eq:100}
    \begin{tikzcd}
      \CHcat^1(C_1 \times C_2) \times \CHcat^1(C_2 \times C_3) \ar[r,"\tilde{\circ}"] \ar[d,"\tilde{T}\times \tilde{T}"'] &
      \CHcat^1(C_1 \times C_3) \ar[d,"\tilde{T}"] \\
     \stHom(\CHcat^1(C_2), \CHcat^1(C_1)) \times  \stHom(\CHcat^1(C_3), \CHcat^1(C_2))  \ar[r,""']        &  \stHom(\CHcat^1(C_3), \CHcat^1(C_1))
    \end{tikzcd},
  \end{equation}
    commutes up to natural isomorphisms $\tilde{T}(g\tilde{\circ}{h}) \iso \tilde{T}(g)\circ\tilde{T}(h)$.
\end{lemma}
\begin{proof} For any smooth projective morphism $f: Y\to B$ of relative dimension one and line bundles $L_1, L_2$ on $Y$, let $\<{L
_1, L_2}_{f} = \int_fG_{(L_1, L_2)}$ denote the Deligne line bundle on $B$. Our task is to prove the existence of a natural isomorphism for any $L \in \CHcat^1(C_3)$:
\begin{equation}\label{eq:101}
    \<{\<{\pi_{12}^*\cO(g), \pi_{23}^*\cO(h)}_{\pi_{13}}, \alpha^*_3L}_{\alpha_1} \iso \<{\cO(g), \gamma^*_{2}\<{\cO(h), \beta_3^*L}_{\beta_2}}_{\gamma_1}
\end{equation}
where the maps are
\[\begin{tikzcd}
  C_1\times C_3 \ar[r, "\alpha_3"]  \ar[d, "\alpha_1"] & C_3, & C_2\times C_3 \ar[r, "\beta_3"]  \ar[d, "\beta_2"] &C_3, & C_1\times C_2 \ar[r, "\gamma_2"]  \ar[d, "\gamma_1"] & C_2.\\
  C_1 &&C_2 &&C_1& 
\end{tikzcd} \]
By additivity in $L$, it suffices to consider the case $L = \cO(x)$ for a closed point $x = \Spec~F$ of $C_3$.  We put 
\begin{gather*}
\iota_1\colon C_1 \iso D =C_1 \times x \hookrightarrow C_1 \times C_3, \quad \iota_2\colon C_2 \iso E  = C_2 \times x \hookrightarrow C_2 \times C_3, \\
\iota_{12} \colon C_1 \times C_2 \iso C_1 \times C_2 \times x \hookrightarrow
C_1 \times C_2 \times C_2 = X.
\end{gather*} 
By (\ref{norm-finite}), the left-hand-side of (\ref{eq:101}) is
\[ \<{\<{\pi_{12}^*\cO(g), \pi_{23}^*\cO(h)}_{\pi_{13}}, \alpha^*_3L}_{\alpha_1} \iso N_{D/C_1}(\<{\pi_{12}^*\cO(g), \pi_{23}^*\cO(h)}_{\pi_{13}}\big |_D) \iso \iota_1^*\bigg( \<{\pi_{12}^*\cO(g), \pi_{23}^*\cO(h)}_{\pi_{13}}\bigg).\]
On the other hand, by (\ref{norm-finite}), the right-hand-side of (\ref{eq:101}) is
\[ \<{\cO(g), \gamma^*_{2}\<{\cO(h), \beta_3^*L}_{\beta_2}}_{\gamma_1} \iso \<{\cO(g), \gamma^*_{2}N_{{E}/C_2}(\cO(h)\big |_E)}_{\gamma_1} \iso \<{\cO(g), \gamma^*_{2}\iota_2^*\cO(h) }_{\gamma_1} \iso\<{\cO(g), \iota_{12}^*\pi_{23}^*\cO(h)}_{\gamma_1},\]
using $C_2 \iso E$ for the second isomorphism and the following diagram for the last isomorphism:
\[\begin{tikzcd}
  C_1\times C_2 \ar[d, "\gamma_2"]  \ar[r, "\iso"] &C_1 \times C_2 \times x \ar[r] &C_1 \times C_2 \times C_3 =X \ar[d, "\pi_{23}"]\\
   C_2 \ar[r, "\iso"] & E \ar[r]& C_2 \times C_3\,,
\end{tikzcd}
\]
where the top row is $\iota_{12}$ and the bottom one is $\iota_2$. 
As $\pi_{12}\circ \iota_{12}$ is the identity map on $C_1 \times C_2$, we have
\[
\<{\cO(g), \iota_{12}^*\pi_{23}^*\cO(h)}_{\gamma_1} \iso \<{\iota_{12}^* \pi_{12}^*\cO(g), \iota_{12}^*\pi_{23}^*\cO(h)}_{\gamma_1}.
\]  
The required natural isomorphism in (\ref{eq:101}), namely, 
\[\<{\iota_{12}^*\pi_{12}^*\cO(g), \iota_{12}^*\pi_{23}^*\cO(h)}_{\gamma_1} \iso \iota_1^*\bigg( \<{\pi_{12}^*\cO(g), \pi_{23}^*\cO(h)}_{\pi_{13}}\bigg) 
\] 
follows from functoriality: use the map of relative curves
\[\begin{tikzcd}
  C_1\times C_2 \ar[d, "\gamma_1"]  \ar[r, "\iso"] &C_1 \times C_2 \times x \ar[r] &C_1 \times C_2 \times C_3 =X \ar[d, "\pi_{13}"]\\
   C_1 \ar[r, "\iso"] & D \ar[r]& C_1 \times C_3\,,
\end{tikzcd}
\]
where the top row is still $\iota_{12}$ and the bottom one is now $\iota_1$.
This proves Lemma \ref{compos-eh?}. 
\end{proof}
Taking $C=C_1 = C_2=C_3$ in the above lemma, we obtain that 
\[ \tilde{\circ}: \CHcat^1(C\times C) \times \CHcat^1(C \times C) \to \CHcat^1(C \times C)\] 
is a monoidal functor of Picard categories and that
\[\tilde{T}: \CHcat^1(C \times C) \to \stEnd(\CHcat^1(C))\] is a functor of ring categories proving (iii).

This finishes the proof of Theorem~\ref{nuovo}.

\bibliographystyle{hamsalpha}
\bibliography{Deligne-bundle}

\end{document}